\documentclass[aop,MSNbibl,seceqn,dvips]{arximspdf}
\usepackage{graphicx}

\doi{10.1214/12-AOP761} 
\volume{41}
\issue{6}
\pubyear{2013}
\firstpage{4116}
\lastpage{4161}

\makeatletter

\newcommand{\HF}{{HF}}
\newcommand{\FH}{{FH}}

\newcommand{\AF}{\mathit{AF}}
\newcommand{\doot}{.}

\newcommand{\bbar}{\overline}

\newcommand{\cddots}{\cdots}
\newcommand{\rrvert}{\vert}
\newcommand{\llvert}{\vert}

\newcommand{\xrightarrow}[1]{\stackrel{#1}{\longrightarrow}}

\renewcommand{\a}{\alpha}
\renewcommand{\b}{\beta}
\newcommand{\G}{\Gamma}
\newcommand{\D}{\Delta}
\renewcommand{\l}{\lambda}
\newcommand{\m}{\mu}
\renewcommand{\o}{\omega}
\newcommand{\s}{\sigma}
\renewcommand{\t}{\tau}
\renewcommand{\th}{\theta}
\newcommand{\F}{\Phi}

\newcommand{\Ac}{{\mathcal A}}
\newcommand{\Oc}{{\mathcal O}}
\newcommand{\E}{{\mathbb E}}
\newcommand{\N}{{\mathbb N}}
\newcommand{\Z}{{\mathbb Z}}

\newtheorem{theorem}{Theorem}[section]
\newtheorem{proposition}[theorem]{Proposition}
\newtheorem{lemma}[theorem]{Lemma}
\newtheorem{corollary}[theorem]{Corollary}

\newtheorem{properties}[theorem]{Properties}
\newtheorem{fact}[theorem]{Fact}

\newproclaim{definition}[theorem]{Definition}
\newproclaim{remark}[theorem]{Remark}

\newproclaim{examples}[theorem]{Examples}
\newproclaim{question}[theorem]{Question}

\makeatother

\begin{document}
\begin{frontmatter}

\title{Behaviors of entropy on finitely generated groups}
\runtitle{Behaviors of entropy}

\begin{aug}
\author[A]{\fnms{J\'er\'emie} \snm{Brieussel}\corref{}\thanksref{t1}\ead[label=e1]{jeremie.brieussel@gmail.com}\ead[label=u1,url]{https://www.math.kyoto-u.ac.jp/\textasciitilde brieussel/}}
\runauthor{J. Brieussel}
\affiliation{Kyoto University}
\address[A]{Department of Mathematics\\
Kyoto University\\
Kitashirakawa Oiwake-cho \\
Sakyo-ku Kyoto-shi \\
606-8502 Kyoto\\
Japan \\
\printead{e1}\\
\printead{u1}} 
\end{aug}

\thankstext{t1}{Supported by Swiss NSF Grant 20-126689.}

\received{\smonth{10} \syear{2011}}
\revised{\smonth{3} \syear{2012}}

%
\begin{abstract}
A variety of behaviors of entropy functions of random walks on finitely
generated groups is presented, showing that for any $\frac{1}{2} \leq
\alpha\leq\beta\leq1$, there is a group $\G$ with measure $\m$
equidistributed on a finite generating set such that
\[
\liminf\frac{\log H_{\G,\m}(n)}{\log n}=\a,\qquad
\limsup \frac{\log H_{\G,\m}(n)}{\log n}=\b.
\]
The groups involved are finitely generated subgroups of the group of
automorphisms of an extended rooted tree. The return probability and
the drift of a simple random walk $Y_n$ on such groups are also
evaluated, providing an example of group with return probability
satisfying
\[
\liminf\frac{{\log}|{\log P}(Y_n=_\G1)|}{\log n}=\frac{1}{3},\qquad
\limsup\frac{{\log}|{\log P}(Y_n=_\G1)|}{\log n}=1
\]
and drift satisfying
\[
\liminf\frac{\log\E\|Y_n\|}{\log n}=\frac{1}{2},\qquad
\limsup\frac{\log\E\|Y_n\|}{\log n}=1.
\]
\end{abstract}

%

\end{frontmatter}

\section{Introduction}

The characterization of groups by an asymptotic description of their
Cayley graphs may be dated back to Folner's criterion of amenability by
quasi-invariant subsets \cite{Fol}. Shortly after, Kesten showed
equivalence with the probabilistic criterion that a group $\G$ is
nonamenable if and only if the return probability $P(Y_n=1)$ of a
simple random walk $Y_n$ on $\G$ decays exponentially fast~\cite{Kes}.

This article focuses on three quantities that partially describe the
behavior of the diffusion process of a random walk $Y_n$ with step
distribution $\m$ on a group $\G$. Namely the entropy function
$H_{\G,\m}(n)=H(\m^{\ast n})=H(Y_n)$, the return probability $P(Y_n=1)$ and
the drift, also called rate of escape, $L_{\G,\m}(n)=\E_{\m^{\ast
n}}\|\gamma\|=\E\|Y_n\|$, where \mbox{$\|\cdot\|$} is a word norm,
\[
H(\m)=-\sum_{\gamma\in\G} \m(\gamma) \log\m(\gamma)
\]
is the Shannon entropy of the probability measure $\m$ and $\m^{\ast
n}$ is the $n$-fold convolution of $\m$, or, in other terms, the
distribution of $Y_n$. The return probability for a finitely supported
symmetric law $\m$ is a group invariant by \cite{PSC1}, which is not
known to be the case for entropy and drift. However, sublinearity of
the entropy or of the drift for some measure $\m$ with generating
support implies amenability \cite{KV}. In the present paper, the
measure $\m$ will always be equidistributed on a canonical finite
symmetric generating set of~$\G$.

The asymptotic behavior of these functions has been precisely
established in a number of cases that mainly include virtually
nilpotent groups, linear groups and wreath products
\cite{Var,PSC2,PSC3,Ers3,Ersdrift,Ers4,Rev} and a variety of less precise estimates exist for some groups
acting on rooted trees \cite{BV,Kai,Ers2,BKN,Bri2,AAV,AV,Zuk}.

The object of this article is to present examples of groups that
provide new asymptotic behaviors for these probabilistic functions.
Entropy, return probability and drift functions will not be precisely
computed, but only mild approximations in terms of their exponents.
%
\begin{definition}
The \textit{lower} and \textit{upper entropy exponents} of a random walk
$Y_n$ of law $\m$ on a group $\G$ are, respectively,
\[
\underline{h}(\G,\m)=\liminf\frac{\log H_{\G,\m}(n)}{\log n}
\quad\mbox{and}\quad \overline{h}(\G,
\m)=\limsup\frac{\log H_{\G,\m
}(n)}{\log n}.
\]
The \textit{lower} and \textit{upper return probability exponents} of a
random walk $Y_n$ of symmetric finitely supported law $\m$ on a group
$\G$ are, respectively:
\[
\underline{p}(\G)=\liminf\frac{{\log}|{\log P}(Y_n=1)|}{\log n}
\quad\mbox{and}\quad \overline{p}(\G)=
\limsup\frac{{\log}|{\log
P}(Y_n=1)|}{\log n}.
\]
The \textit{lower} and \textit{upper drift exponents} of a random walk $Y_n$
of law $\m$ on a group $\G$ are, respectively,
\[
\underline{\delta}(\G,\m)=\liminf\frac{\log L_{\G,\m}(n)}{\log
n} \quad\mbox{and}\quad \overline{
\delta}(\G,\m)=\limsup\frac{\log
L_{\G,\m}(n)}{\log n}.
\]
When equality holds, the \textit{entropy exponent} of the group $\G$ with
law $\m$ is $h(\G,\m)=\overline{h}(\G,\m)=\underline{h}(\G,\m
)$, the \textit{return probability exponent} of a group $\G$ is $p(\G
)=\underline{p}(\G)=\overline{p}(\G)$ and the \textit{drift exponent}
of the group $\G$ with law $\m$ is $\delta(\G,\m)=\overline
{\delta}(\G,\m)=\underline{\delta}(\G,\m)$. Return probability
exponents do not depend on the particular choice of the measure by
\cite{PSC1}.
\end{definition}
Computing the exponents gives moderate information on the function. By
\cite{PSC2}, the wreath product $\Z\wr\Z$ has return probability
$P(Y_n=1) \approx\break\exp(-n^{{1/3}}(\log n)^{{2/3}})$, and
the lamplighter $F \wr\Z$ with finite group $F$ has return
probability $P(Y_n=1) \approx\exp(-n^{{1/3}})$. Both have
return probability exponent $\frac{1}{3}$. Exponent $1$ does not imply
linearity of entropy or drift, nor exponential decay of return
probability, as seen by the exemples in \cite{Ers2}.

The groups considered here are directed groups of automorphisms of
extended rooted trees. The main construction combines the directed
groups of \cite{Bri2} with the notion of boundary permutational
extension introduced in \cite{BE} and used in \cite{Bri3} to exhibit
various behaviors of growth functions on groups. The entropy exponents
can be computed explicitely in terms of the group construction (see
Theorem \ref{MTE}), which leads to the following corollary:
%
\begin{theorem}\label{MT}
For any $\frac{1}{2} \leq\a\leq\b\leq1$, there exists a finitely
generated group $\G=\G(\a,\b)$ and a symmetric finitely supported
measure $\m$ such that
\[
\underline{h}(\G,\m)=\a\quad\mbox{and}\quad \overline{h}(\G,\m)=\b.
\]
In particular when $\a=\b$, there is a finitely generated group $\G
(\b)$ with measure $\m$ such that
\[
h(\G,\m)=\b.
\]
\end{theorem}

Entropy is related to growth because over a finite set, entropy is
maximized for equidistribution probability, so that
\[
h_{\G,\m}(r)=h\bigl(\m^{\ast n}\bigr)\leq\log\# \operatorname{supp}\bigl(
\m^{\ast n}\bigr)=\log b_{\G,S}(r).
\]
The groups of Theorem \ref{MT} have sublinear entropy and often
exponential growth. However, most of the groups $\G_\o$ of
\cite{Bri3} have intermediate growth and are extended directed groups of a
binary rooted tree, so by Theorem \ref{MTE}, they all have entropy
exponent $h(\G,\m)=\frac{1}{2}$.

The return probability and drift exponents of extended directed groups
can be estimated from above and from below, but the bounds do not match
in general. A~specific example provides the following theorem:
%
\begin{theorem}\label{MTT}
There exists a finitely generated group $\D$ and a symmetric finitely
supported measure $\m'$ such that
\[
\underline{p}(\D)=\tfrac{1}{3},\qquad \overline{p}(\D)=1
\quad\mbox{and}\quad
\underline{\delta}\bigl(\D,\m'\bigr)=\tfrac{1}{2},\qquad \overline{
\delta}\bigl(\D,\m'\bigr)=1.
\]
\end{theorem}

These theorems show that the phenomenon of oscillation studied in
\cite{Bri3} for growth function exponents (see also \cite{Bri1} and
\cite{KP}) also occurs for entropy, return probability and drift. The
existence of a group with such drift exponents was mentioned without
proof in \cite{Ers6}.

The article is structured as follows. The main construction of extended
directed groups of a rooted tree $T_{\bar d}$ is described in Section
\ref{involved}, where is also presented a side application to the
Haagerup property of groups with nonuniform growth. Section~\ref{srp}
presents the basic tool to study these groups, which is the rewriting
process, leading to the notions of minimal tree and activity, related
to inverted orbits of permutational extensions defined in \cite{BE}.
Section \ref{sectionRW} relates the expected activity of a random walk
$Y_n$ to the exponent sequence, which depends only on the tree $T_{\bar
d}$. At this stage, one can prove the main Theorem \ref{MTE} on
entropy, which implies Theorem \ref{MT} for $\b<1$ and allows us to
derive estimates on the drift. The frequency of oscillation of entropy
exponents is also studied in Section \ref{smte}. The main estimates on
return probability of extended directed groups are given in Theorem
\ref{RPE}, with a specific example related to the lamplighter group.
Section \ref{highorder} is devoted to another construction adapted
from \cite{KP} and similar to \cite{Ers2}, which allows us to obtain
the case $\b=1$ in Theorem \ref{MT} and to prove Theorem \ref{MTT}.
A~generalization of the construction of extended directed groups is
presented in Section \ref{gls}, followed by some comments and
questions in the final Section~\ref{cq}.

\section{The groups involved}\label{involved}

\subsection{Directed groups}\label{directedgroups} Given a sequence
$\bbar{d}=(d_j)_{j \geq0}$ of integers $d_j \geq2$, the
\textit{spherically homogeneous rooted tree} $T_{\bar{d}}$ is the graph with
vertices $v=(i_1 i_2 \cddots i_k)$ with $i_j$ in $\{1,2,\ldots,d_{j-1}\}
$, including the empty sequence $\varnothing$ called the root, and edges
$\{(i_1 i_2 \cddots i_k),(i_1 i_2 \cddots i_k i_{k+1})\}$. The index $k$ is
called the \textit{depth} or \textit{level} of $v$, denoted $|v|=k$.

The boundary $\partial T_{\bar d}$ of the tree $T_{\bar d}$ is the
collection of infinite sequences $x=(i_1 i_2 \cddots)$ with $i_j$ in $\{
1,2,\ldots,d_{j-1}\}$.

The group $\operatorname{Aut}(T_{\bar d})$ of automorphisms of the rooted tree is the
group of graph automorphisms that fix the root $\varnothing$. The
following isomorphism is canonical:
%
\begin{equation}
\label{iso} \operatorname{Aut}(T_{\bar{d}}) \simeq \operatorname{Aut}(T_{ \s\bar{d}}) \wr
S_{d_0}.
\end{equation}
The symbol $\wr$ represents the permutational wreath product $G\wr
S_d=(G\times\cdots\times G) \rtimes S_d$ where $S_d$ acts on the
direct product of $d$ copies of $G$ by permutation, and $\s$
represents the shift on sequences, so that $\s\bbar d=(d_1,d_2,\ldots
)$. As the isomorphism (\ref{iso}) is canonical, we identify an
element and its image and write
%
\begin{equation}
\label{wreathim} g=(g_1,\ldots,g_{d_0})\s
\end{equation}
with $g$ in $\operatorname{Aut}(T_{\bar{d}})$, the $g_i$ in $\operatorname{Aut}(T_{\s\bar{d}})$
and $\s$ in $S_{d_0}$. The automorphism $g_t$ represents the action of
$g$ on the subtree $T_t$, isomorphic to $T_{\s\bar d}$, hanging from
vertex~$t$, and the rooted component $\s$ describes how these subtrees
$(T_t)_{t=1\cddots d_0}$ are permuted.

With notation (\ref{wreathim}), for any vertex $ty$ in $T_{\bar d}$,
one has $g(ty)=\s(t)g_t(y)$. If $f=(f_1,\ldots,f_{d_0})\t$, then
$(gf)(ty)=(f \circ g)(ty)=\t(\s(t)) f_{\s(t)}(g_t(y))=\break(\s\t
)(t)(g_tf_{\s(t)})(y)$, so that $gf=(g_1f_{\s(1)},\ldots,g_{d_0}f_{\s
(d_0)})\s\t$.

An automorphism $g$ is \textit{rooted} if $g=(1,\ldots,1)\s$ for a
permutation $\s$ in $S_{d_0}$. The group of rooted automorphisms of
$T_{\bar d}$ is obviously isomorphic to $S_{d_0}$ and can be realized
canonically as a subgroup or a quotient of $\operatorname{Aut}(T_{\bar d})$.

The set $H_{\bar d}$ of automorphisms \textit{directed by the geodesic
ray} $1^\infty=(111\cddots)$ in $T_{\bar d}$ is defined recursively. An
element $h$ is in $H_{\bar d}$ if there exists $h'$ in $H_{\s\bar d}$
and $\s_2,\ldots,\s_{d_0}$ rooted in $\operatorname{Aut}(T_{\s\bar d})$ such that
%
\begin{equation}
\label{defh} h=\bigl(h',\s_2,\ldots,\s_{d_0}
\bigr).
\end{equation}
There is a canonical isomorphism of abstract groups, $H_{\bar d} \simeq
S_{d_1} \times\cdots\times S_{d_1} \times H_{\s\bar d}$ with $d_0-1$
factors $S_{d_1}$. As a consequence, $H_{\bar d}$ is the uncountable
but locally finite product
%
\begin{equation}
\label{Habstrait} H_{\bar d} \simeq S_{d_1} \times\cdots\times
S_{d_1} \times S_{d_2} \times\cdots\times S_{d_2}
\times\cdots
\end{equation}
with $d_{l-1}-1$ factors $S_{d_l}$, indexed by $\{2,\ldots,d_{l-1}\}$.
Under this isomorphism, denote $h=(\s_{1,2},\ldots,\s_{1,d_0},\s_{2,2},\ldots,\s_{2,d_1},\ldots)$ with $\s_{k,t}$ in $S_{d_k}$.

The action of $h \in H_{\bar d}$ on the rooted tree $T_{\bar d}$ and
its boundary $\partial T_{\bar d}$ is given by
\[
h\bigl(1^{k-1}t i_{k+1}i_{k+2}\cddots
\bigr)=1^{k-1}t\s_{k,t}(i_{k+1})i_{k+2}
\cddots,
\]
where each vertex or boundary element is uniquely written $1^{k-1}t
i_{k+1}i_{k+2}\cddots$ with $t$ in $\{2,\ldots,d_{k-1}\}$ and $k\geq1$
integer. The notation $1^k$ is a shortcut for $11\cddots1$ with $k$ terms.
%
\begin{definition}
A group $G$ of automorphisms of $T_{\bar d}$ is called \textit{directed}
when it admits a generating set of the form $S \cup H$ where $S$ is
included in the group $S_{d_0}$ of rooted automorphisms, and $H$ is
included in $H_{\bar d}$. Denote by $G=G(S,H)$ such a directed group.
Then $G(S_{d_0},H_{\bar d})$ is the (uncountable) \textit{full directed
group} of $T_{\bar d}$. Say a group $G(S,H)$ is \textit{saturated} if
$S=S_{d_0}$ is the full group of rooted automorphism, and $H$ is a
group such that the projection of the equidistribution measure over $H$
onto each factor $S_{d_l}$ in (\ref{Habstrait}) is the
equidistribution measure on~$S_{d_l}$. The full directed group is
obviously saturated.
\end{definition}

Assume the sequence $\bbar d=(d_i)_i$ is bounded taking finitely many
values $e_1,\ldots,e_T$, then the direct product $H=S_{e_1}\times\cdots
\times S_{e_T}$ embeds diagonally into $S_{d_1} \times\cdots\times
S_{d_1} \times\cdots\simeq H_{\bar d}$ (where the factors $S_{e_t}$
embeds diagonaly into the subproduct of factors for which $d_l=e_t$).
With the obvious identifications, the group $G(S_{d_0},H)$ is a
finitely generated saturated directed group. Note that this precise
group is minimal among saturated directed groups of~$T_{\bar d}$.

\subsection{Extended directed groups} For a fixed $x$, equip the
finite set $xT_d=\{x,x1,\ldots,xd\}$ with a structure of rooted
tree\vadjust{\goodbreak}
with root $x$ and one level $\{x1,\ldots,\allowbreak xd\}$. The \textit{extended
boundary} $E\partial T_{\bar d}$ of the rooted tree $T_{\bar d}$ is
obtained by replacing each boundary point $x$ by a short rooted tree
$xT_d$:
\[
E\partial T_{\bar d}=\{xT_d|x \in\partial T_{\bar d}
\}=\bigl\{(x;x1,\ldots,xd) |x \in\partial T_{\bar d}\bigr\}.
\]
Call an \textit{extended tree} the set $ET_{\bar d}=T_{\bar d} \sqcup
E\partial T_{\bar d}$. Its group of automorphisms is the group
%
\begin{equation}
\label{pwp} \operatorname{Aut}(ET_{\bar d}) =S_d \wr_{\partial T_{\bar d}}
\operatorname{Aut}(T_{\bar d})=\{\varphi\dvtx\partial T_{\bar d} \rightarrow
S_d\} \rtimes \operatorname{Aut}(T_{\bar d}),
\end{equation}
where the action of the group $\operatorname{Aut}(T_{\bar d})$ on functions is given
by $g\doot\varphi(x)=\varphi(gx)$ so that $(g_1g_2).\varphi
=g_2\doot(g_1\doot\varphi)$. The group $\operatorname{Aut}(ET_{\bar d})$ of automorphisms of
an extended tree was introduced by Bartholdi and Erschler in \cite{BE}
as a ``permutational wreath product over the boundary.'' The wreath
product isomorphism (\ref{iso}) extends well:
%
\begin{proposition}\label{CEpourET}
There is a canonical isomorphism
\[
\operatorname{Aut}(ET_{\bar d}) \simeq \operatorname{Aut}(ET_{\s\bar d}) \wr_{\{1,\ldots, d_0\}}
S_{d_0}.
\]
\end{proposition}
\begin{pf}
Any $\gamma$ in $\operatorname{Aut}(ET_{\bar d})$ is decomposed $\gamma=\varphi g$,
with $g \in \operatorname{Aut}(T_{\bar d})$ and $\varphi\dvtx\partial T_{\bar d}
\rightarrow S_d$. The classical isomorphism (\ref{iso}) provides a
decomposition $g=(g_1,\ldots,g_d)\s$. Also the boundary of the tree
can be decomposed into $d_0$ components $\partial T_{\bar d} = \partial
T_1 \sqcup\cddots\sqcup\partial T_d$ with $T_t \simeq T_{\s\bar d}$
the tree descended from the first level vertex $t$. Set $\varphi_t=\varphi|_{\partial T_t}$ the restriction of $\varphi$. With this
notation, the application $\F$ realizing the canonical isomorphism is
given by
\[
\F(\gamma)=(\varphi_1 g_1,\ldots, \varphi_d
g_d)\s\in \operatorname{Aut}(ET_{\s
\bar d}) \wr_{\{1,\ldots, d_0\}}
S_{d_0}.
\]
In order to prove the proposition, it is sufficient to check that $\F
(\gamma\gamma')=\F(\gamma)\F(\gamma')$.

On the one hand, $\gamma\gamma'=\varphi g\varphi' g'=\varphi
(g\doot\varphi') gg'= \psi gg'$, with $\psi=\varphi(g\doot\varphi')$. As
above set $\psi_t=\psi|_{\partial T_t}$, and as classicaly
$gg'=(g_1g'_{\s(1)},\ldots,g_d g'_{\s(d)})\s\s'$, the embedding is
\[
\F\bigl(\gamma\gamma'\bigr)=\bigl(\psi_1
g_1 g'_{\s(1)},\ldots,\psi_d
g_d g'_{\s
(d)}\bigr)\s\s'.
\]
On the other hand,
\begin{eqnarray*}
\F(\gamma)\F\bigl(\gamma'\bigr) &=& (\varphi_1g_1,
\ldots, \varphi_dg_d)\s \bigl(\varphi_1'g_1',
\ldots,\varphi_d'g_d'\bigr)
\s'
\\
&=& \bigl(\varphi_1g_1 \varphi_{\s(1)}'g_{\s(1)}',
\ldots,\varphi_dg_d \varphi_{\s(d)}'g_{\s(d)}'
\bigr)\s\s'
\\
&=& \bigl(\varphi_1\bigl(g_1\doot\varphi_{\s(1)}'
\bigr) g_1g_{\s(1)}',\ldots,\varphi_d
\bigl(g_d\doot\varphi_{\s(d)}'\bigr)
g_dg_{\s(d)}'\bigr)\s\s'.
\end{eqnarray*}
There remains to check $\psi_t=\varphi_t(g_t\doot\varphi_{\s(t)}')$,
and indeed for any $y \in\partial T_t \simeq\partial T_{\s\bar d}$,
\begin{eqnarray*}
\psi_t(y) &=& \psi(ty)=\bigl(\varphi\bigl(g\doot\varphi'
\bigr)\bigr) (ty)=\varphi (ty) \bigl(\bigl(g\doot\varphi'\bigr) (ty)
\bigr)=\varphi(ty)\varphi'(g\doot ty)
\\
&=& \varphi(ty) \varphi'\bigl(\s(t) (g_t\doot y)\bigr)=
\varphi_t(y)\varphi_{\s
(t)}'(g_t\doot y)=
\varphi_t(y) \bigl(g_t\doot \varphi_{\s(t)}'
\bigr) (y).
\end{eqnarray*}
\upqed
\end{pf}\eject

Functions over $\partial T_{\bar d}$ supported on $1^\infty=1111\cddots
$ will play a specific role. For $f \in S_d$, denote by $\varphi_f
\dvtx\partial T_{\bar d} \rightarrow S_d$ the function $\varphi_f(1^\infty)=f$ and $\varphi_f(x)=id$ if $x \neq1^\infty$. Note
that for $h$ in $H_{\bar d}$, one has $h \varphi_f= \varphi_f h$ in
$\operatorname{Aut}(ET_{\bar d})$ because $h$ fixes the ray $1^\infty$.
%
\begin{definition}
A group $\G$ of automorphisms of the extended tree $ET_{\bar d}$ is
called \textit{directed} when it admits a generating set of the form $S
\cup H \cup F$ where $S$ is rooted, $H$ is included in $H_{\bar d}$ and
elements of $F$ have the form $\varphi_f$ for $f \in S_d$. Denote by
$\G=\G(S,H,F)$ such a directed group. Say a group $\G
(S_{d_0},H,F)<\operatorname{Aut}(ET_{\bar d})$ is \textit{saturated} if $G(S_{d_0},H)$ is
saturated. Saturation implies that equidistribution on $H \times F$
projects to equidistribution on each factor $S_{d_k}$ of (\ref
{Habstrait}) and on the factor $F$ which justifies the notation \mbox{$\G
(S_{d_0},H\times F)$} for saturated directed groups.
\end{definition}

Unless mentioned otherwise, use for the directed group $\G(S,H,F)$ the
set $S \cup H \times F$, where both $S$ and $H \times F$ are finite
groups themselves (hence symmetric) in the case of finitely generated
saturated directed groups.

Denote by $H_l$ the restriction of $H<H_{\bar d}$ to levels $\geq l$,
that is, the projection of $H$ to $H_{\s^l \bar d}=S_{d_{l+1}} \times
\cdots\times S_{d_{l+1}} \times\cdots\,$. Also denote by $S_l$ the
subgroup of $S_{d_l}$ generated by the projections of $H$ on the
$d_{l-1}-1$ factors $S_{d_l}$ of (\ref{Habstrait}).
%
\begin{proposition}\label{wpd} Let $\G(S,H,F)<\operatorname{Aut}(ET_{\bar d})$ be a
directed group. Then there is a canonical embedding
\[
\G(S,H,F) \hookrightarrow\G(S_1,H_1,F) \wr S.
\]
More generally,
\[
\G(S_l,H_l,F) \hookrightarrow\G(S_{l+1},H_{l+1},F)
\wr S_l.
\]
\end{proposition}
\begin{pf}
The embedding is clear from Proposition \ref{CEpourET} because the
images of the generators are given by $s=(1,\ldots,1)s$, $h=(h',\s_2,\ldots,\s_d)$ and $\varphi_f=(\varphi_f,1,\ldots,1)$.
\end{pf}

Observe that if $\G(S,\HF)$ is saturated, then $\G(S_l,H_lF)$ is
saturated for all $l$.

\subsection{Examples} The class of directed groups of $\operatorname{Aut}(T_{\bar
d})$ contains many\allowbreak examples of groups that have been widely studied in
relation with torsion \cite{Ale,Gri1,Gri2,GS,BS}, intermediate growth
\cite{Gri1,Gri2,BS,Ers}, subgroup growth
\cite{Seg}, nonuniform
exponential growth \cite{Wil1,Wil2,Bri2} and
amenability \mbox{\cite{GZ,BV,Bri2,BKN,AAV,AV}}.

Theorem 3.6 in \cite{Bri2} states that the full directed group
$G(S_{d_0},H_{\bar d})$ is amenable if and only if the sequence $\bbar
d$ is bounded. This result obviously extends to the setting of
automorphisms of extended trees. Indeed, the group $\G(S,H,F)$ is a
subgroup of $F \wr_{\partial T_{\bar d}} G(S,H)$ which is a group
extension of a direct sum of finite\vadjust{\goodbreak} groups (copies of $F$) by the group
$G(S,H)$, hence is amenable when valency $\bbar d$ is bounded. On the
other hand, $G(S,H)$ is a quotient of $\G(S,H,F)$, so the latter
inherits nonamenability when $\bbar d$ is unbounded.

The notion of automorphisms of extended trees is a reformulation of
permutational wreath product, which was introduced in \cite{BE} in
order to compute explicit intermediate growth functions; see also \cite
{Bri3}. The boundary extension $T_d$ does not need to be a finite tree;
that is, the permutationnal wreath product $F \wr_{\partial T_{\bar
d}}$ makes sense for any group $F$. This was used in \cite{BE} to make
a stack of extensions of trees and compute the growth function of some
finitely generated groups of their automorphisms, namely $b_k(r)
\approx e^{r^{\a_k}}$, where $\a_k \rightarrow1$ with the number $k$
of extensions in the stack.

\subsection{Nonuniform growth and Haagerup property} This paragraph
illustrates the interest of extended trees by an application that will
not be used further in the rest of the article. It can be omitted at
first reading.

A group is said to have the Haagerup property if it admits a proper
continuous affine action on a Hilbert space; see \cite{CCJJV}. This
is, for instance, the case for free groups and amenable groups. Groups
with the Haagerup property have attracted interest as they are known to
satisfy the Baum--Connes conjecture.

Denote by $\Ac_d<S_d$ the group of alternate permutations, and given a
bounded sequence $\bbar d$ of integers $d \leq d_i \leq D$, define the
finite subgroup $A_{\bar d}<H_{\bar d}<\operatorname{Aut}(T_{\bar d})$ by the following:
\begin{longlist}[(2)]
\item[(1)] abstractly, $A_{\bar d} \simeq\Ac_d \times\cdots\times
\Ac_D$ with projection on factor $d'$ denoted $\mathrm{pr}_{d'}$;

\item[(2)] an element $b \in A_{\bar d}$ is realized in $\operatorname{Aut}(T_{\bar
d})$ according to the recursive rule $b=(b',\mathrm{pr}_{d_1}(b),1,\ldots,1)$
for $b' \in A_{\s\bar d}<\operatorname{Aut}(T_{\s\bar d})$.
\end{longlist}
Take $F$ to be the free product $\Ac_5 \ast\Ac_5$, and consider the
group $\G(\Ac_{d_0},A_{\bar d},\Ac_5 \ast\Ac_5)<\operatorname{Aut}(ET_{\bar d})$
which is a directed group of automorphisms of an extended tree (with
infinite extension at the boundary so that $\Ac_5 \ast\Ac_5<S_\infty$).
%
\begin{proposition}\label{prop2.5}
If $29 \leq d_i \leq D$, the group $\G(\Ac_{d_0},A_{\bar d},\Ac_5
\ast\Ac_5)<\break\operatorname{Aut}(ET_{\bar d})$ has nonuniform exponential growth, is
nonamenable and has\break Haagerup property.
\end{proposition}

Directed groups of the form $G(\Ac_{\bar d},A_{\bar d})<\operatorname{Aut}(T_{\bar
d})$ were introduced by Wilson as the first examples of groups with
nonuniform exponential growth \cite{Wil2,Wil1}. For bounded
valency and $A_{\bar d}$ finite as above, these groups are amenable
\cite{Bri2} and hence have Haagerup property. For unbounded valency
and $A_{\bar d}\simeq\Ac_5 \ast\Ac_5$, these groups are
nonamenable. Zuk asked whether they still have Haagerup property. The
groups $\G(\Ac_{d_0},A_{\bar d},\Ac_5 \ast\Ac_5)$ resemble the
groups $G(\Ac_{\bar d},A_{\bar d})$ for unbounded sequences $\bbar d$
in the sense that the ``free product factor'' $\Ac_5 \ast\Ac_5$ is
``located at the boundary of the tree.''\vadjust{\goodbreak} Thus the proposition above
hints that Wilson groups $G(\Ac_{\bar d},A_{\bar d})$ of \cite{Wil2}
have Haagerup property (but it does not prove it). It provides the
first example of groups of nonuniform growth for which Haagerup
property does not follow from amenability. No example of group having
nonuniform growth but not Haagerup property is known.
\begin{pf*}{Proof of Proposition \ref{prop2.5}}
Each of the groups $\Ac_{d_0}$, $A_{\bar d}$ and $\Ac_5 \ast\Ac_5$
is perfect and generated by finitely many involutions (the number of
which depends only on $D$), so that this is also the case for $\G(\Ac_{d_0},A_{\bar d},\Ac_5 \ast\Ac_5)$. Moreover, there is an isomorphism
\[
\G(\Ac_{d_0},A_{\bar d},\Ac_5 \ast
\Ac_5) \simeq\G(\Ac_{d_1},A_{\s\bar d},
\Ac_5 \ast\Ac_5) \wr\Ac_{d_0}.
\]
(By Proposition 7.2 in \cite{Bri2}, $\G$ contains any commutator
$([b_1,b_2],1,\ldots,1)$ for $b_1,b_2$ a pair of elements in one of the
generating groups $\Ac_{d_0},A_{\bar d},\Ac_5 \ast\Ac_5$. By
perfection, it shows that the embedding of Proposition \ref{wpd} is onto.)

This shows that $\G(\Ac_{d_0},A_{\bar d},\Ac_5 \ast\Ac_5)$ belongs
to a class $\chi$ (see \cite{Wil2}) and there exists generating sets
with growth exponent tending to $1$. Also, the group contains a
nontrivial free product, hence is nonamenable, and has nonuniform
exponential growth.

The group $\G(\Ac_{d_0},A_{\bar d},\Ac_5 \ast\Ac_5)$ inherits
Haagerup property for it is contained in $(\Ac_5 \ast\Ac_5) \wr_{\partial T_{\bar d}} G(\Ac_{d_0},A_{\bar d})$ which itself has
Haagerup property. Indeed, there is a short exact sequence
\[
1 \rightarrow(\Ac_5 \ast\Ac_5)^{\partial T_{\bar d}}
\rightarrow (\Ac_5 \ast\Ac_5) \wr_{\partial T_{\bar d}} G(
\Ac_{d_0},A_{\bar
d}) \rightarrow G(\Ac_{d_0},A_{\bar d})
\rightarrow1,
\]
where $(\Ac_5 \ast\Ac_5)^{\partial T_{\bar d}}$ has the Haagerup
property and $G(\Ac_{d_0},A_{\bar d})$ is amenable, which implies the
Haagerup property for the middle term (Example 6.1.6 in~\cite{CCJJV}).
\end{pf*}

In the remaining sections of the present article (as well as in \cite
{Bri3}), only finite groups $F$ are considered, which simplifies some
arguments and still allows us to observe various phenomena.

\section{Rewriting process and activity of words}\label{srp}

\subsection{Rewriting process} Given a finitely generated saturated
(extended) directed group $\G(S_{d_0},\HF)$ acting on a tree $T_{\bar
d}$ of bounded valency, a canonical generating set is $S_{d_0} \sqcup
\FH$. Note that by saturation and bounded valency, both $S_{d_0}$ and
$\HF$ are finite subgroups of $\G(S_{d_0},\HF)$. In the present article,
the disjoint union of these subgroups is taken as a generating set. In
particular both $1_{S_{d_0}}$ and $1_{\mathit{HF}}$ are generators, distinct in
a ``word perspective.''
%
\begin{definition}
Given the generating set $S_{d_0} \sqcup \FH$ of a finitely generated
saturated group $\G(S_{d_0},\HF)$, a representative word is \textit{alternate} if it has the form $w=s_1k_1s_2k_2\cddots s_nk_ns_{n+1}$ for
some $s_i$ in $S_{d_0}$ and $k_i$ in $\HF$.\vadjust{\goodbreak}
\end{definition}

Note that any word $w$ in this generating set admits a \textit{canonical
alternate} form obtained by merging packs of successive terms belonging
to the same finite group $S_{d_0}$ or $\HF$. For example, the canonical
alternate form of $s_1s_2k_3k_3^{-1}1_{S_{d_0}}k_41_{\mathit{HF}}$ is
$(s_1s_2)1_{\mathit{FH}}1_{S_{d_0}}k_4$.

The alternate form $w=s_1k_1\cddots s_nk_ns_{n+1}$ of a word in the
generating set $S_d \sqcup \FH$ is equivalent to $w=k_1^{\s_1}\cddots
k_n^{\s_n}\s_{n+1}$, for $\s_i=s_1\cddots s_i$, where $k^\s=\s k\s^{-1}$ denotes the conjugate of $k$ by $\s$. Remember that $k_j$ in
$\HF$ is uniquely decomposed into $k_j=(k'_j,b_{j,2},\ldots,b_{j,d_0})$
with $k'_j$ in $H_1F$ and $b_{j,t}$ in $S_{d_1}$ rooted.
%
\begin{proposition}[(Rewriting process)]\label{rp}
With the notation above, the alternate word $w=s_1k_1\cddots
s_nk_ns_{n+1}$ can be algorithmically rewritten in the wreath product
as
\[
w=\bigl(w^1,\ldots,w^{d_0}\bigr)\s_\varnothing,
\]
where:
\begin{longlist}[(4)]
\item[(1)] the rooted permutation is $\s_\varnothing=s_1s_2\cddots
s_{n+1}=\s_{n+1}$;

\item[(2)] the word $w^t=s_1^tk_1^t\cddots
s_{m_t}^tk_{m_t}^ts_{m_t+1}^t$ is alternate in the generating system
$S_{d_1} \sqcup H_1F$ of the saturated directed group $\G
(S_{d_1},H_1F)$, of length $m_t \leq\frac{n+1}{2}$ such that
$m_1+\cdots+m_{d_0} \leq n$;

\item[(3)] the factors $s_i^t$ depend only on factors $b_{j,t'}$ at
times $j$ when $\s_j(t)=t'$;
\item[(4)] the factors $k_i^t$ depend only on factors $k_j'$ at times
$j$ when $\s_j(t)=1$.
\end{longlist}
\end{proposition}
\begin{pf}
The factor $k_j^{\s_j}=\s_jk_j \s_j^{-1}$ has an image in the wreath product
\[
k_j^{\s_j}=\bigl(b_{j,\s_j(1)},\ldots,k_j',
\ldots,b_{j,\s_j(d_0)}\bigr)
\]
with $k_j'$ in position $\s_j^{-1}(1)$. This shows that $w^t$ in the
wreath product image of $w$ is a product of terms $b_{j,\s_j(t)}$ at
times $j$ when $\s_j(t) \neq1$ and terms $k_j'$ at times $j$ when $\s_j(t)=1$. The word $w^t$ is the canonical alternate form of this
product. Each factor $k_j$ furnishes a factor $k_j'$ to exactly one of
the coordinates, so the sum of length is $\leq n$. Also by looking at
alternate forms, there has to be a factor $k_j^{\s_j}$ such that $\s_j(t)\neq1$, giving $b_{j,t'}$ on coordinate $t$, between two factors
such that $\s_j(t)=1$, giving $k_j'$ on coordinate $t$, so that the
length of $w^t$ is at most half the length of $w$.
\end{pf}

Of course, the rewriting process can be iterated, and at each vertex,
$v=ut$ is associated to the alternate word $w$, another alternate word
$w^v$ of length $m_v$ in the generating set $S_{d_{|v|}} \sqcup
H_{|v|}F$ of the saturated directed group $\G(S_{d_{|v|}}, H_{|v|})$
defined inductively by $w^v=w^{ut}=(w^u)^t$. Note that the word $w^v$
has length $m_v \leq\frac{n}{2^{|v|}}+1$, and $\sum_{|v|=l} m_v \leq n$.

\subsection{Minimal tree and activity}\label{mta}
%
\begin{definition}
The \textit{minimal tree} $T(w)$ of the alternate word $w$ in $S_{d_0}
\sqcup \HF$ is the minimal regular rooted subtree of $T_{\bar d}$ such
that $m_v \leq1$ for any leaf $v$ of $T(w)$. Recall that a subtree $T$
of $T_{\bar d}$ is regular if whenever a vertex $v$ belongs to $T$,
either all its descendants $vt$ also belong to $T$, or none of them does.
\end{definition}

The minimal tree $T(w)$ is constructed algorithmically. Indeed it
contains the root $\varnothing$, and if $v$ is in $T(w)$, either $m_v
\leq1$ and no descendant of $v$ belongs to $T(w)$, or $m_v \geq2$ and
all the sons of $v$ belong to $T(w)$. As $m_v$ decays exponentially
fast with generations, the minimal tree $T(w)$ has depth at most $\log_2 n$.

The leaves of the minimal tree $T(w)$ satisfy either $m_v=0$, in which
case they are called \textit{inactive}, or $m_v=1$, called \textit{active},
leaves. The subset $A(w)$ of the boundary (set of leaves) $\partial
T(w)$ of the minimal tree $T(w)$ is called the \textit{active} set of the
word $w$. Its size is the \textit{activity} $a(w)=\# A(w)$ of the word
$w$. The activity of a word $w=s_1$ of length $0$ (no factor in $\HF$)
is $a(s)=0$.

\begin{remark}
The minimal tree $T(w)$, as well as the set of active leaves $A(w)$, of
a word $w=s_1k_1\cddots k_ns_{n+1}$ depend only on the word $s_1h_1\cddots
h_ns_{n+1}$ where $k_j=\varphi_{f_j}h_j$, that is, on the quotient
$G(S,H)$ of $\G(S,\HF)$.
\end{remark}

The notions of minimal tree and active set allow an interesting
description of the action of a word $w$ on the tree $T_{\bar d}$.
Indeed, the action of the automorphism $\gamma=_{\G}w$ in
$\operatorname{Aut}(ET_{\bar d})$ is determined by the following data:
\begin{longlist}[(3)]
\item[(1)] the minimal tree $T(w)$ and its active set $A(w)$;
\item[(2)] the permutations $\s_u \in S_{d_{|u|}}$ attached to
vertices $u \in T(w)\setminus A(w)$;\vspace*{1pt}
\item[(3)] the ``short words'' $w^v=s_1^vk^vs_2^v$ attached to active
vertices $v$.
\end{longlist}
The description of the short word $w^v=s_1^vk^vs_2^v$ at an active
vertex $v$ can be refined into a tree action $s_1^vh_vs_2^v$, together
with a boundary function given by $\varphi_{f^v}(x)=f^v$ for
$x=v(s_1^v(1))1^\infty$ and $\varphi_{f_v}(x)=1$ otherwise.

A point $x$ in $\partial T_{\bar d}$ is called active when it is of the
form $x=v(s_1^v(1))1^\infty$ for some active leaf $v$ of $T(w)$. If
the element $\gamma=_\G w$ in $\G(S,\HF)$ has the form $\gamma
=\varphi g$ in the permutational wreath product (\ref{pwp}), then the
support of $\varphi$ is included in the set $A_\partial(w)$ of active
boundary points; see Figure~\ref{fig1}.
%
\begin{figure}

\includegraphics{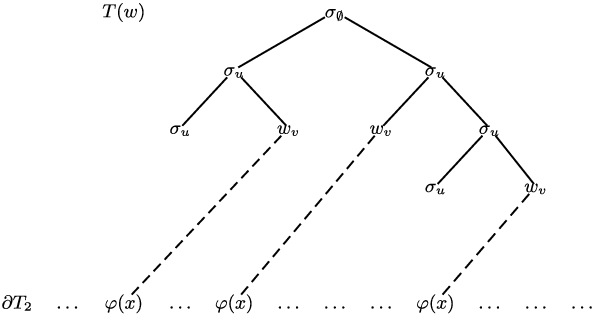}

\caption{Description of the action of a word $w$ via the minimal tree $T(w)$.}
\label{fig1}
\end{figure}

\subsection{Ascendance forest}\label{ascfor} Given an alternate word
$w$ in $\G(S_{d_0},\HF)$ and the collection $(w^v)_{v \in T(w)}$ of
words obtained by rewriting process, each word $w^v$ is a product of
factors $w^v=s_1^vk_1^v\cddots k_{m_v}^vs_{m_v+1}^v$. Recall that each
factor $k_i^{vt}$ in the word $w^{vt}$ is a product of terms $k'^v_j$,
obtained from $(k_j^v)^{\s_j^v}=(b^v_{j,\s^v_j(1)},\ldots, k'^v_j,
\ldots,\allowbreak b^v_{j,\s^v_j(d_{|v|})})$ with $k'^v_j$ in position $t$, during\vspace*{-1pt}
the rewriting process of the word $w^v$ for the vertex~$v$, into
$w^v=(w^{v1},\ldots,w^{vd_{|v|}})\s_v$.

Consider the graph $\AF(w)$ with set of vertices the collection
$(k_i^v)_{v,i}$ of factors in $H_{|v|}F$ appearing in the rewritten
words $(w^v)_{v \in T(w)}$, where a pair of factors
$(k_i^v,k_{j}^{v'})$ is linked by an edge when $v'=vt$ and the term
$k'^v_i$ appears in $k_{j}^{vt}$.

%
\begin{fact}
The graph $\AF(w)$ is a forest, that is, a graph with no loop. More
precisely, $\AF(w)$ is a union of trees $(\t_v)_{v \in A(w)}$ indexed
by the active set of $w$. Moreover, if $v$ is an active leaf, then
\[
k^v=\prod_{k_j \in\partial\t_v}k^{(|v|)}_j
\]
is the product ordered with $j$, where $k^{(|v|)}_j$ is the restriction
of $k_j \in \HF$ to $H_{|v|}F$. In particular,
\[
f_v=\prod_{k_j \in\partial\t_v}f_j.
\]
\end{fact}
\begin{pf}
Let $v=t_1\cddots t_l$ be an active leaf with $w^v=s_1^vk^vs_2^v$. The
term\vspace*{1pt} $k^v$ is a product of terms $k'^{t_1\cddots t_{l-1}}_j$ in
$w^{t_1\cddots t_{l-1}}$ [for those\vspace*{2pt} $j$'s where $\s_j^{t_1\cddots
t_{l-1}}(1)=t_l$], which are the neighbors of the vertex $k^v$ in the
graph $\AF(w)$. This describes the ball of center $k^v$ and radius $1$.

Inductively, each factor $k_i^{t_1\cddots t_{l-r}}$ which represents a
vertex in the sphere of center $k^v$ and radius $r$ is a product of
terms $k'^{t_1\cddots t_{l-r-1}}_j$ in $w^{t_1\cddots t_{l-r-1}}$ [for
those $j$'s where $\s_j^{t_1\cddots t_{l_r-1}}(1)=t_{l-r}$], which form
the link\vspace*{-1pt} of $k_i^{t_1\cddots t_{l-r}}$ in the sphere of radius $r+1$.

Now\vspace*{-1pt} each factor $k'^{t_1\cddots t_\l}_j$ in a word of level $\l$
contributes to exactly one factor $k_i^{t_1\cddots t_\l t_{\l+1}}$,
which rules out the possibility of a loop.

So the connected component of $k^v$ in the graph $\AF(w)$ is a tree $\t_v$, it is also the $l$-ball of center $k^v$ and the leaves of $\t_v$
form the $l$-sphere, which is precisely the set of factors $k_j$ of $w$
that lie in $\t_v$. By construction of $\AF(w)$, $k^v$ is the required
ordered product.
\end{pf}
%
\begin{remark}\label{remarkphi}
Observe that the ascendance forest $\AF(w)$ of a word $w=s_1k_1\cddots
k_ns_{n+1}$ depends only on the word $s_1h_1\cddots h_n s_{n+1}$ for
$k_j=\varphi_{f_j}h_j$. Indeed given a factor $k_i^{t_1\cddots t_\l}$,
its link to factors in level $\l+1$ is determined by the factor $\s_i^{t_1 \cddots t_\l}$, which is determined by the factors $s_1^{t_1
\cddots t_\l},\ldots, s_i^{t_1 \cddots t_\l}$ themselves determined by the
sequence $(h_j^{t_1\cddots t_{\l-1}})_j$. Thus the link of $k_i^v$ does
not depend on the sequence $(f_j)_{j \in\{1,\ldots,n\}}$. A
consequence of this observation and the preceding fact is that for any
function $\varphi$ with support included in $A_\partial(w)$, there
exists $(f_j)_{j \in\{1,\ldots,n\}}$ such that $\varphi s_1h_1\cddots
h_n s_{n+1}=_\G s_1 \varphi_{f_1}h_1 \cddots\varphi_{f_n}h_n s_{n+1}$.
Moreover the number of such $n$-tuples is independent of the function
$\varphi$. This shows:
\end{remark}
%
\begin{fact}\label{cond}
Given $g=s_1h_1\cddots h_n s_{n+1}$, for any $\varphi\dvtx\partial T_{\bar
d} \rightarrow F$ with support in $A_\partial(g)$, one has
\[
\#\{f_1,\ldots,f_n \in F| \varphi g =_\G
s_1\varphi_{f_1}h_1\cddots \varphi_{f_n}h_ns_{n+1}
\}= \biggl(\frac{1}{\#F} \biggr)^{\#
A_\partial(g)}.
\]
\end{fact}

\subsection{Counting activity} The inverted orbit of a product
$w=r_1\cddots r_k$ of elements $r_i$ of $\operatorname{Aut}(T_{\bar d})$ is the set $\{
1^\infty,r_k 1^\infty,r_{k-1}r_k1^\infty,\ldots, r_1\cddots
r_k1^\infty\}$, denoted by $\Oc(w)$; see \cite{BE}. The inverted
orbit of $w^{-1}$ coincides with the activity set $A_\partial(w)$
defined in Section \ref{mta}. The notion of activity is classical in
the context of automorphisms of rooted trees; see \cite{Nek,Sid}.

%
\begin{proposition}\label{counta}
Let $\G(S_{d_0},\HF)$ be a finitely generated saturated directed group
acting on a tree of bounded valency $\bbar d$. The activity function
$a(w)$, which for $w=s_1 \varphi_{f_1}h_1 \cddots\varphi_{f_n}h_n
s_{n+1}$ counts equivalently:
\begin{longlist}[(4)]
\item[(1)] the size of the set $A(w)$ of active leaves in the minimal
tree $T(w) \subset T_{\bar d}$;
\item[(2)] the size of the set $A_\partial(w)$ of active boundary
points in $\partial T_{\bar d}$;
\item[(3)] the number of trees (i.e., connected components) in the
ascendance forest $\AF(w)$;
\item[(4)] the size of the inverted orbit $\Oc
(s_{n+1}^{-1}h_n^{-1}\cddots h_1^{-1}s_1^{-1})$ in the sense of \cite{BE},
satisfies under rewriting process $w=(w^1,\ldots,w^{d_0})\s_\varnothing$:
\[
a(w)=a\bigl(w^1\bigr)+\cdots+a\bigl(w^{d_0}\bigr).
\]
Moreover, the constraint $a(w) \leq r$ only allows us to describe at
most exponentially many elements in $\G(S_{d_0},\HF)$; that is, there
exists a constant $C$ depending only on $D=\max\{d_i\}$ and $\#\HF$
such that
\[
\#\bigl\{\gamma\in\G(S_{d_0},\HF)|\exists w=_\G\gamma,
a(w) \leq r\bigr\} \leq C^r.
\]
\end{longlist}
\end{proposition}

Note that $a(w)$ is the activity for words in the group
$\G(S_{d_0},\HF)$, and $a(w^t)$ is the activity function for words in
the group $\G(S_{d_1},H_1F)$.
\begin{pf*}{Proof of Proposition \ref{counta}}
Points (1), (2), (3) are clear from the descriptions above. Point (4)
is shown by induction on $n$. Suppose\break $s_1k_1\cddots k_{n-1}s_n=\varphi_n g_n$ for
$g_n=s_1h_1\cddots h_{n-1}s_n$, then $s_1k_1\cddots
k_{n-1}s_n\varphi_{f_n}h_n=\varphi_n(g_n\doot\varphi_{f_n})g_nh_n$ and
the point $g_n^{-1}(1^\infty)$ which is the support of the function
$g_n\doot\varphi_{f_n}$ is added to the set of active leaves $A_\partial(w)$.

The equality on activities under rewriting process is trivial when $n
\leq1$, and if $n \geq2$ then the minimal tree is not restricted to
the root, one has $T(w)=T(w^1) \sqcup\cddots\sqcup T(w^{d_0}) \sqcup\{
\varnothing\}$, so that $A(w)=A(w^1) \sqcup\cddots\sqcup A(w^{d_0})$,
and equality holds.

In order to prove the exponential bound, first observe that $\#\partial
T(w) \leq D a(w)$. It is obvious if $\partial T(w)$ is the root $\{
\varnothing\}$ or the first level $\{1,\ldots,d_0\}$, and then clear for
arbitrary $T(w)$ by induction on the size $\#T(w)$ since $\partial
T(w)=\partial T(w^1) \sqcup\cddots\sqcup\partial T(w^{d_0})$.

Now if $a(w) \leq r$, its minimal tree $T(w)$ has a boundary of size
$\leq Dr$, so there are $\leq K^r$ possibilities for $T(w)$ (for some
$K$ depending only on $D$). To finish the description of $\gamma=_\G
w$, one has to choose permutations $\s_u$ at vertices $u \in T(w)
\setminus\partial T(w)$, for which there are $\leq D!^{Dr}$
possibilities, and short words $s_1^vk_vs_2^v$ at leaves $v \in
\partial T(w)$, for which there are $\leq(D!^2\#\HF)^{Dr}$ possibilities.
\end{pf*}

\section{Random walks}\label{sectionRW}

Given a finitely generated saturated directed group $\G(S_{d_0},\HF)$,
consider random alternate words $Y_n=s_1k_1s_2k_2\cddots s_nk_ns_{n+1}$,
where $s_i$ in $S_{d_0}$ and $k_i$ in $\HF$ are equidistributed, and all
factors are independent. Such a random alternate product $Y_n$ is the
simple random walk on $\G(S_{d_0},\HF)$ for the symmetric generating
set $S_{d_0}\HF S_{d_0}$ (for the product of two independent
equidistributed variables $s'_is_{i+1}$ in $S_{d_0}$ is another
equidistributed variable, independent of other factors).

\subsection{Inheritance of random process through wreath product}
Given a random alternate word $Y_n$, the rewriting process of
Proposition \ref{rp} furnishes an image in the wreath product
$Y_n=(Y_n^1,\ldots,Y_n^{d_0})\t_n$. Each coordinate $Y_n^t$ is a
random process on words in $S_{d_1} \sqcup H_1F$ by Proposition \ref
{wpd}, which turns out to be a random alternate word of random length.
%
\begin{lemma}\label{randomRP}
Denote $Y_n=(Y_n^1,\ldots,Y_n^{d_0})\t_n$ the alternate words obtained
by rewriting process of a random alternate word $Y_n=s_1k_1s_2k_2\cddots
s_nk_ns_{n+1}$ in a finitely generated saturated directed group $\G
(S_{d_0},\HF)$. Then:
\begin{longlist}[(3)]
\item[(1)] The rooted random permutation is $\t_n=s_1\cddots s_{n+1}$
and hence is equidistributed.
\item[(2)] The random length $m_t$ of the random product $Y_n^t$ has
the law of the sum of $n$ independent Bernoulli variables $(u_j)$ on $\{
0,1\}$ with $P(u_j=1)=p_0=\frac{d_0-1}{d_0^2}$.

In particular, by the law of large numbers $m_t \sim p_0n$ almost
surely, and by the principle of large deviations, for any $\theta>0$,
there exists $c_\theta<1$ such that
\[
P \biggl(\frac{m_t}{n} \notin[p_0-\theta,p_0+
\theta] \biggr) \leq c_\theta^n.
\]
\item[(3)] For each coordinate $t$, the conditioned variable
$Y_n^t|m_t$ has precisely the law of the random alternate word
$Y'_{m_t}$ of length $m_t$ in $S_{d_1}\sqcup H_1F$, that is,
\[
Y_n^t=Y'_{m_t}=s_1^tk_1^t
\cddots k_{m_t}^ts_{m_t+1}^t,
\]
where the factors $s_j^t$ and $k_j^t$ are equidistributed in $S_{d_1}$
and $H_1F$, respectively (except $s_1^t$ and $s_{m_t+1}^t$), and all
factors $(s_j^t,k_j^t)_j$ are independent.
\end{longlist}
\end{lemma}

This lemma is a restating of Lemma 4.6 in \cite{Bri1}. A brief proof is
given below.
\begin{pf*}{Proof of Lemma \ref{randomRP}}
As in Proposition \ref{rp}, write $Y_n=k_1^{\s_1}\cddots k_n^{\s_n}\s_{n+1}$, where each factor has an image in the wreath product $k_j^{\s
_j}=(b_{j,\s_j(1)},\ldots,\break k_j',\ldots,\allowbreak b_{j,\s_j(d_0)})$ with $k'_j
\in H_1F$ in position $\s_j(t)$ and $b_{j,s} \in S_{d_1}$. As $(\s_j)_{j=1}^n$ is a sequence of independent terms equidistributed in
$S_{d_0}$ [for $\s_j=s_1\cddots s_j$ with $(s_i)_{i=1}^n$ is a sequence
of independent terms equidistributed in $S_{d_0}$], the position
sequence $(\s_j(t))_{j=1}^n$ is equidistributed in $\{1,\ldots,d_0\}$
for any choice of $t$.

Then for a fixed $t$, $Y_n^t$ is a product of $n$ terms which are
either $b_{j,\s_j(t)}$ at times $j$ when $\s_j(t) \neq1$, which
happens with probability $\frac{d_0-1}{d_0}$, or $k'_j$ at times $j$
when $\s_j(t)=1$, which happens with probability $\frac{1}{d_0}$. In
both cases, the factors are equidistributed in $S_{d_1}$ or $H_1F$,
respectively, because $k_j$ is equidistributed in $\HF$ by saturation.
Moreover, all the terms are independent.

Now to obtain an alternate word, the runs of successive terms that
belong to the same finite group (either $S_{d_1}$ or $H_1F$) are
merged, the factors $(s_j^t,k_j^t)_j$ are still equidistributed and independent.

There remains to count the number of such runs, given by
\[
2m_t+1=1+\sum_{j=1}^n
1_{\{(\s_j(t)=1\ \mathrm{and}\ \s_{j+1}(t)
\neq1)\ \mathrm{or}\ (\s_j(t)\neq1\ \mathrm{and}\ \s_{j+1}(t) =
1)\}}.
\]
Knowing\vspace*{1pt} that $P(\s_j(t)=1)=\frac{1}{d_0}$ and $P(\s_j(t)\neq
1)=\frac{d_0-1}{d_0}$ independently of previous terms, two successive
terms belong to different finite groups with probability $2\frac
{d_0-1}{d_0}\frac{1}{d_0}=2p_0$.\vspace*{1pt}

Note that $m_t$ depends only on $\s_1,\ldots,\s_n$, that is, on
$s_1,\ldots,s_n$, whereas the factors $(s_j^t,k_j^t)$ are determined by
$k_1,\ldots,k_n$. In particular, fixing $\s_1,\ldots,\s_n$ (hence~$m_t$), and playing with $k_1,\ldots,k_n$ any alternate word $Y_n^t$ of
length $m_t$ in $S_{d_1}\sqcup H_1F$ appears with the same probability.
\end{pf*}

The lemma can be iterated to show that for any vertex $v$ in $T_{\bar
d}$, the random product $Y_n^v$ obtained by rewriting process of the
random alternate word $Y_n$ is also an alternate random word in the
group $\G(S_{d_{|v|}},H_{|v|}F)$, of length $m_v \sim p_0\cddots
p_{|v|}n$ almost surely, that is, the conditioned variable
$Y_n^v|m_v=Y_{m_v}^{(|v|)}$ is a random alternate word in $\G
(S_{d_{|v|}},H_{|v|}F)$ of length $m_v$.

\subsection{Exponent sequence associated to valency sequence}\label
{defbeta} Given a bounded sequence $\bbar d=(d_i)_i$ of integers $\geq
2$, define $\bbar p=(p_i)_i$ by $p_i=\frac{d_i-1}{d_i^2}$. Note $d=\min
(d_i)$, $D=\max(d_i)$, $p=\max(p_i)=\frac{d-1}{d^2}$ and $P=\min
(p_i)=\frac{D-1}{D^2}$. Define the \textit{exponent function} $\b(n)$
associated to the valency sequence $\bbar d$ by
\[
k(n)=k_{\bar d}(n)=\min\{k|p_0\cddots p_k n \leq1
\}
\quad\mbox{and}\quad
\b (n)=\b_{\bar d}(n)=\frac{\log(d_0\cddots d_{k(n)})}{\log n}.
\]

Moreover, given a small $\th\neq0$ and an integer $N_0$ depending
only on $\theta$, set
\[
k^\th(n)=\min\bigl\{k|(p_0+\th)\cddots(p_k+
\th)n \leq N_0\bigr\} \quad\mbox{and}\quad \b^\th(n)=
\frac{\log(d_0\cddots d_{k^\th(n)})}{\log n}.
\]
For $N_0=1$, the function $k^\th(n)$ is increasing with $\th$, and
$\b^\th(n) \longrightarrow_{\th\rightarrow0} \b(n)$ for a fixed $n$.
%
\begin{proposition}\label{thetapprox}
There exists a function $\varepsilon(\th) \longrightarrow_{\th
\rightarrow0} 0$ such that for all $n$ large enough (depending on $\th$),
\[
\bigl|\b(n)-\b^\th(n)\bigr| \leq\varepsilon(\th).
\]
\end{proposition}
\begin{pf} Assume $\th>0$ (similar proof for $\th<0$). For $n$ large
enough, $k^\th(n) \geq k(n)$, and the difference is
\begin{eqnarray*}
\bigl|\b(n)-\b^\th(n)\bigr| &=& \bigl|{\log_n} (d_0\cddots
d_{k(n)})-\log_n (d_0\cddots d_{k^\th(n)})\bigr|
\\
&=& \log_n (d_{k(n)+1}\cddots d_{k^\th(n)})
\end{eqnarray*}
and hence is bounded by
\[
\bigl|\b(n)-\b^\th(n)\bigr| \leq\bigl|k^\th(n)-k(n)\bigr|
\frac{\log D}{ \log n}.
\]
By the definition of $k(n)$ and $k^\th(n)$, there are inequalities
\[
\frac{N_0}{n} \geq(p_0+\th)\cddots(p_{k^\th(n)}+\th) >
\frac
{N_0(p_{k^{\th}(n)}+\th)}{n} \geq p_0\cddots p_{k(n)}N_0(P+
\th),
\]
so that
\begin{eqnarray*}
(p+\th)^{k^\th(n)-k(n)} &\geq& (p_{k(n)+1}+\th )\cddots(p_{k^\th(n)}+
\th)
\\
&=& \frac{(p_0+\th)\cddots(p_{k^\th
(n)}+\th)}{(p_0+\th)\cddots(p_{k(n)}+\th)}
\\
&\geq& \biggl( \frac{p_0}{p_0+\th} \biggr) \cddots \biggl( \frac
{p_{k(n)}}{p_{k(n)}+\th}
\biggr)N_0(P+\th)
\\
&\geq& \biggl( \frac{P}{P+\th} \biggr)^{k(n)}N_0(P+
\th),
\end{eqnarray*}
which shows that for some constant $K$,
\[
\bigl|k^\th(n)-k(n)\bigr| \leq k(n) \biggl\llvert \frac{\log({P}/({P+\th
}))}{\log(p+\th)} \biggr
\rrvert +K.
\]
Notice that $ \frac{\log n}{|{\log P}|} \leq k(n) \leq\frac{\log
n}{|{\log p}|}$ to finally obtain
\[
\bigl|\b^\th(n)-\b(n)\bigr| \leq\frac{\log D}{|{\log p}|}\biggl\llvert
\frac{\log
({P}/({P+\th}))}{\log(p+\th)} \biggr\rrvert + \frac{K'}{\log n}.
\]
The proposition follows from the limit $|{\log}(\frac{P}{P+\th})
|\longrightarrow_{\th\rightarrow0} 0$.
\end{pf}
%
\begin{fact}\label{bornesup} There is a bound on $\b(n)$ depending
only on the bounds $d \leq d_i \leq D$ on the valency of the rooted
tree $T_{\bar d}$. Namely for some constant $C$ depending only on $D$,
\[
\b_d \leq\b(n) \leq\b_D+\frac{C}{\log n},
\]
where $\b_d=\frac{\log d}{- \log p}=\frac{\log d}{ \log(
{d^2}/({d-1}))}=\frac{1}{1+{\log({d}/({d-1}))}/{\log d}}=\frac
{1}{2-{\log(d-1)}/{\log d}}$.
\end{fact}

Note that $\b_2=\frac{1}{2}$ and $\b_d \longrightarrow_{d
\rightarrow
\infty} 1$.
\begin{pf*}{Proof of Fact \ref{bornesup}}
Note that
\begin{eqnarray*}
\b(n)&=&\frac{\log(d_0\cddots d_{k(n)})}{\log n} \geq \frac{\log(d_0\cddots d_{k(n)})}{-\log(p_0 \cddots p_{k(n)})},
\\
\b(n) &\leq& \frac{\log(d_0\cddots d_{k(n)})}{-\log(p_0 \cddots
p_{k(n)})+\log p_{k(n)}}\leq\frac{\log(d_0\cddots d_{k(n)})}{-\log
(p_0 \cddots p_{k(n)})}+\frac{C}{\log n}.
\end{eqnarray*}
As $p_i=\frac{d_i-1}{d_i^2}$, simply compute
\begin{eqnarray*}
&&\frac{\log(d_0 \cddots d_{k(n)})}{\log({d_0^2}/({d_0-1})\cddots
{d_{k(n)}^2}/({d_{k(n)}-1}))} \\
&&\qquad= 1\Big/\biggl({1+\frac{\log(
{d_0}/({d_0-1}))+\cdots+ \log({d_{k(n)}}/({d_{k(n)}-1})) }{\log
d_0+\cdots+\log d_{k(n)}}}\biggr) \\
&&\qquad\leq\frac{1}{1+{\log(
{D}/({D-1}))}/{\log D}},
\end{eqnarray*}
and similarly for the lower bound, by the inequality on ratio of average
\begin{eqnarray*}
\frac{\log({D}/({D-1}))}{\log D} &\leq&\frac{\log(
{d_0}/({d_0-1}))+\cdots+ \log({d_{k(n)}}/({d_{k(n)}-1})) }{\log
d_0+\cdots+\log d_{k(n)}}\\
&\leq&\frac{\log({d}/({d-1}))}{\log d}.
\end{eqnarray*}
\upqed
\end{pf*}

For a real number $x>0$, define $h_{\bar d}(x)=d_0\cddots d_{k(x)}$ for
the unique integer $k(x)$ such that $\frac{1}{p_0}\cddots\frac
{1}{p_{k(x)}}\geq x > \frac{1}{p_0}\cddots\frac{1}{p_{k(x)-1}}$. In
particular, $n^{\b_{\bar d}(n)}=h_{\bar d}(n)$ for any integer $n$.
%
\begin{fact}\label{classfunction}
Given a nondecreasing function $g(x)$ such that
\[
dg(x) \leq g \biggl(\frac{d^2}{d-1}x \biggr)
\quad\mbox{and}\quad g \biggl(
\frac{D^2}{D-1}x \biggr)\leq Dg(x),
\]
there exists a sequence $\bbar d$ in $\{d,D\}^\N$ and a constant $C$
such that
\[
\frac{1}{C} g(x) \leq h_{\bar d}(x) \leq Cg(x).
\]
\end{fact}
\begin{pf}
Set $x_k=\frac{1}{p_0}\cddots\frac{1}{p_{k}}$ so that
$h(x_{k+1})=d_0\cddots d_{k+1}=h(x_k)d_{k+1}$ and $x_{k+1}=\frac
{1}{p_{k+1}}x_k$. Assume by induction that $d_0,\ldots,d_k$ are
constructed. Then:
\begin{longlist}[(2)]
\item[(1)] if $\frac{h(x_k)}{g(x_k)}\geq1$, set $d_{k+1}=d$, and obtain
\[
\frac{dh(x_k)}{Dg(x_k)} \leq\frac{dh(x_k)}{g(x_k/P)} \leq \frac{h(x_{k+1})}{g(x_{k+1})} =
\frac{h(x_k/p)}{g(x_k/p)} \leq\frac{dh(x_k)}{dg(x_k)};
\]
\item[(2)] if $\frac{h(x_k)}{g(x_k)} < 1$, set $d_{k+1}=D$, and obtain
\[
\frac{Dh(x_k)}{Dg(x_k)} \leq\frac{Dh(x_k)}{g(x_k/P)}= \frac{h(x_{k+1})}{g(x_{k+1})}\leq
\frac{Dh(x_k)}{g(x_k/p)}\leq\frac{Dh(x_k)}{dg(x_k)}.
\]
\end{longlist}
This shows that $\frac{d}{D} \leq\frac{h(x_k)}{g(x_k)} \leq\frac
{D}{d}$. As moreover $\frac{P}{p} \leq\frac{x_k}{x_{k+1}} \leq\frac
{p}{P}$ and $g$ is nondecreasing, this proves the fact.
\end{pf}

\subsection{Expected activity} The expectation of activity is ruled by
the exponent sequence.
%
\begin{lemma}\label{expectact}
For any $\th>0$ and $n$ large enough, there exists $C_\th>0$ such that
\[
\frac{1}{C_\th} n^{\b^{-\th}(n)} \leq\E a(Y_n) \leq
C_\th n^{\b
^\th(n)},
\]
where the functions $\b(n),\b^\th(n)$ are defined in Section \ref
{defbeta}. In particular, for any $\varepsilon>0$ and $n$ large enough,
\[
\biggl\llvert \frac{\log\E a(Y_n)}{\log n}-\b(n)\biggr\rrvert \leq\varepsilon.
\]
\end{lemma}

For the following proof observe the fact that for any words $a(ww')
\geq a(w)$ by Proposition \ref{counta}(4), so that the function $\E
a(Y_n)$ is nondecreasing with $n$.
\begin{pf*}{Proof of Lemma \ref{expectact}}
By Proposition \ref{counta}, the activity of $Y_n$ relates to activity
on the inherited process by $a_0(Y_n) = a_1(Y_n^1)+\cdots
+a_1(Y_n^{d_0})$, where $a_k(w)$ is the activity function on the group
$\G(S_{d_k},H_kF)$. Thus
\[
\E a_0(Y_n) = \sum_{t=1}^{d_0}
\E a_1\bigl(Y_n^t\bigr).
\]
Now by Lemma \ref{randomRP}, the conditioned variable $Y_n^t|m_t$ is a
random alternate word $Y'_{m_t}$ of random length $m_t$ in the group
$\G(S_{d_1},H_1F)$. Compute by conditioning and the large deviation principle.
\begin{eqnarray*}
\E a_1\bigl(Y_n^t\bigr) &=& \sum
_{i=0}^n \E\bigl(a_1
\bigl(Y_n^t\bigr)|m_t=i\bigr)P(m_t=i)
\\
&\leq& \sum_{i \leq(p_0+\theta)n}P(m_t=i) \E
a_1(Y_i) +nP\bigl(m_t \geq
(p_0+\theta)n\bigr)
\\
&\leq& \E a_1\bigl(Y'_{(p_0+\theta)n}\bigr)+n
c_\theta^n,
\end{eqnarray*}
where $c_\theta<1$ [use that $\E a_1(Y'_n)$ is nondecreasing to bound
the sum]. This shows that for $N_0$ large enough so that $d_0N_0c_\th^{N_0}\leq1$ and $n \geq N_0$,
\[
\E a_0(Y_n) \leq d_0 \E a_1
\bigl(Y'_{(p_0+\theta)n}\bigr)+1.
\]

Note that $c_\theta$ and hence $N_0$ depends on $\th$ and on $d_0$,
but as the valency is bounded, they can be chosen uniform for all
$d_i$. This allows us to iterate the above inequality to get for all $k$,
\[
\E a_0(Y_n) \leq d_0\cddots
d_k \E a_{k+1}\bigl(Y^{(k+1)}_{(p_0+\theta)\cddots
(p_k+\theta)n}
\bigr)+d_0\cddots d_{k-1}+\cdots+d_0+1,
\]
provided $n$ is large enough. Recall that $Y_m^{(k)}$ is a random
alternate word of length $m$ in $\G(S_{d_k},H_kF)$. For $k(n)=\min\{k
|(p_0+\theta)\cddots(p_k+\theta)n\leq N_0 \}$, and this shows
\[
\E a_0(Y_n) \leq d_0\cddots
d_{k(n)}\bigl(\E a_{k^\th(n)+1}\bigl(Y^{(k^\th
(n)+1)}_{N_0}
\bigr)+1\bigr) \leq n^{\b^\th(n)}(N_0+1)
\]
by the trivial estimation $a_k(Y_{N_0}) \leq N_0$ for any $k$ and the
definition of $\b^\th(n)$.

Similarly for the lower bound,
\begin{eqnarray*}
\E a_0(Y_n) &\geq& d_0 \E a_1
\bigl(Y'_{(p_0-\th)n}\bigr)-1
\\
&\geq& d_0\cddots d_k \E a_{k+1}
\bigl(Y^{(k+1)}_{(p_0-\th)\cddots(p_k-\th
)n}\bigr) -(d_0\cddots
d_{k-1}+\cdots+1),
\end{eqnarray*}
so that for $k=k^{-\th}(n)=\min\{k|(p_0-\th)\cddots(p_k-\th)n\leq
N_0\}$, one has
\[
\E a_0(Y_n) \geq\tfrac{1}{2} d_0
\cddots d_{k^{-\th}(n)}=\tfrac{1}{2} n^{\b^{-\th}(n)}
\]
by the trivial estimate $\E a_k(Y^{(k)}_N) \geq\frac{3}{2}$ for any $k$
and $N\geq3$ [indeed, $a_k(Y^{(k)}_N) \geq2$ as soon as there are two
distinct elements among the triplet $\{\s_1^{-1}(1),
\s_2^{-1}(1),\allowbreak\s_3^{-1}(1) \}$, which happens with
probability $\geq \frac{3}{4}$ for any value of $d_k$].
\end{pf*}

\section{Entropy exponents}\label{smte}

\subsection{Main theorem} Given a valency sequence $\bbar d=(d_i)_i$
and $p_i=\frac{d_i-1}{d_i^2}$, recall that the exponent sequence $\b_{\bar d}(n)$ is defined by
\[
\b_{\bar d}(n)=\b(n)=\frac{\log(d_0\cddots d_{k(n)})}{\log n}
\qquad\mbox{where }k(n)=\min
\{k|p_0\cddots p_k n \leq1\}.
\]

\begin{theorem}\label{MTE}
Let $\G=\G(S_{d_0},\HF)$ be a finitely generated saturated directed
subgroup of $\operatorname{Aut}(ET_{\bar d})$, $\m$ the measure equidistributed on
$S_{d_0}\HF S_{d_0}$ and $\b(n)$ the exponent sequence of $\bbar d$; then
for any $\varepsilon>0$ and $n$ large enough,
\[
\biggl\llvert \frac{\log H_{\G,\m}(n)}{\log n} -\b(n) \biggr\rrvert \leq \varepsilon.
\]
In particular, $\overline{h}(\G,\m)=\limsup\b(n)$ and $\underline
{h}(\G,\m)=\liminf\b(n)$.
\end{theorem}
Fact \ref{bornesup} ensures that if $d \leq d_i \leq D$ for all $i$, then
\[
\tfrac{1}{2} \leq\b_d \leq\underline{h}(\G,\m) \leq\overline
{h}(\G,\m) \leq\b_D<1.
\]
\begin{pf*}{Proof of Theorem \ref{MTE}}
First prove the upper bound, which is a straightforward generalization
of Proposition 4.11 in \cite{BKN}. Note the similarity with the upper
bound in Lemma \ref{expectact}.\vadjust{\goodbreak} Indeed, under rewriting process
$Y_n=(Y_n^1,\ldots,Y_n^{d_0})\t_n$ one has
\[
H(Y_n) \leq H\bigl(Y_n^1\bigr)+\cdots+H
\bigl(Y_n^{d_0}\bigr)+H(\t_n),
\]
where $\t_n$ is equidistributed on $S_{d_0}$ so $H(\t_n)\leq C$ for a
constant $C$ depending only on $D$.

By Lemma \ref{randomRP}, the law of $Y_n^t|m_t$ under the length
condition $m_t$ is the law of a random alternate product $Y_{m_t}'$ in
the group $\G(S_{d_1},H_1F)$, and the length $m_t$ has the binomial
law of $\sum_{i=1}^n u_i$ for independent $u_i$ in $\{0,1\}$ with
$P(u_i=1)=p_0$, of entropy bounded by $C \log n$ for $C$ depending only
on $D$. This ensures (Lemma A.4 in \cite{BKN})
\begin{eqnarray*}
H\bigl(Y_n^t\bigr) &\leq& \sum
_{m=0}^n H\bigl(Y_m'
\bigr)P(m_t=m) + H(m_t)
\\
&\leq& H\bigl(Y'_{(p_0+\th)n}\bigr)+nP\bigl(m_t
\geq(p_0+\th)n\bigr)+C \log n
\end{eqnarray*}
by splitting the sum at $m=(p_0+\th)n$ for an arbitrary $\th>0$. By
the large deviation principle, there exists $c_\th<1$ such that $P(m_t
\geq(p_0+\th)n) \leq c_\th^n$. Thus there is $N_0$ depending only on
$\th$ and $D$ such that for $n \geq N_0$, one has [for a slightly
larger constant $C$ since $nc_\th^n=o(\log n)$]
\[
H(Y_n) \leq d_0 H\bigl(Y'_{(p_0+\th)n}
\bigr)+C \log n.
\]
As for expected activity, this allows us to integrate the supremum
$H(n)=\sup_{k\geq0} \{H(Y_n^{(k)})\}\leq Cn$, where $Y_n^{(k)}$ is a
random alternate product in the group $\G(S_{d_k},H_kF)$, and $C$ is a
uniform constant depending only on the sizes of the generating sets
$S_{d_k}\sqcup H_kF$, hence only on $D$ and $\#\HF$, into
\[
H(n) \leq d_0\cddots d_k H\bigl((p_0+\th)
\cddots(p_k+\th)n\bigr)+(d_0\cddots d_{k-1}+
\cdots+1)C \log n
\]
as long as $k \leq k^{\th}(n)$, that is, when $(p_0+\th)\cddots
(p_{k-1}+\th)n \geq N_0$; see Section \ref{defbeta} for the
definition of $k^\th(n)$ and $\b^\th(n)$. Thus
\[
H(Y_n) \leq H(n) \leq d_0\cddots d_{k^\th(n)}(CN_0+C
\log n)=n^{\b^\th
(n)}(CN_0+C \log n)
\]
and $H(Y_n) \leq n^{\b(n)+2\varepsilon(\th)}$ for $n$ large enough,
by Proposition \ref{thetapprox}.

To prove the lower bound, the following fact is useful:
%
\begin{fact}\label{supp}
For $\gamma=\varphi g$ in $\G(S_{d_0},\HF)=F \wr_{\partial T}
G(S_{d_0},H)$, one has
\[
P(Y_n=\gamma) \leq \biggl(\frac{1}{\#F} \biggr)^{\# \operatorname{supp}(\varphi)}.
\]
\end{fact}
\begin{pf}
Denote $Y_n=_\G\varphi_n g_n$, and remark that $\operatorname{supp} (\varphi_n)
\subset A_\partial(Y_n)$ of size $a(Y_n)$. Also recall Remark \ref
{remarkphi} that $A_\partial(Y_n)$ depends only on $g_n=\break s_1h_1\cddots
h_ns_{n+1}$, and Fact \ref{cond} that given $A_\partial(g_n)$, any
function $\varphi\dvtx\partial T \rightarrow F$ with support included in
$A_\partial(g_n)$ appears with probability $ (\frac{1}{\#F}
)^{\#A_\partial(g_n)}$.\vadjust{\goodbreak} This allows us to compute by
conditioning on activity.
\begin{eqnarray*}
P(Y_n=\gamma) &=& \sum_{a=1}^n
P\bigl(Y_n=\gamma|a(Y_n)=a\bigr)P\bigl(a(Y_n)=a
\bigr)
\\
&\leq& \sum_{a \geq\#\operatorname{supp}(\varphi)} P\bigl(\varphi_n=
\varphi |a(Y_n)=a\bigr)P\bigl(a(Y_n)=a\bigr)
\\
&\leq& \sum_{a \geq\#\operatorname{supp}(\varphi)} \biggl(\frac{1}{\#F}
\biggr)^a P\bigl(a(Y_n)=a\bigr) \leq \biggl(
\frac{1}{\#F} \biggr)^{\# \operatorname{supp}(\varphi)}.
\end{eqnarray*}
\upqed
\end{pf}

This fact guarantees
\begin{eqnarray*}
H(Y_n)&=&-\sum_{\gamma\in\G} P(Y_n=
\gamma)\log P(Y_n=\gamma)
\\
&\geq& C \sum_{\gamma\in\G} P(Y_n=\gamma)
\#\operatorname{supp}(\varphi)
= C\E_{\m^{\ast n}}\#\operatorname{supp} (\varphi)=C\E\#\operatorname{supp}(\varphi_n),
\end{eqnarray*}
and the expected value of the size of the support of $\varphi_n$
relates to activity.

More precisely by Fact \ref{cond}, given $A_\partial(g_n)$, the
function $\varphi_n\dvtx A_\partial(g_n) \rightarrow F$ is random, so that
\[
\E\bigl[\#\operatorname{supp}(\varphi_n)|A_\partial(g_n)\bigr]=
\frac{\#F-1}{\#F}\#A_\partial (g_n).
\]
This allows us, once again, to compute by conditioning on activity.
\begin{eqnarray*}
\E\#\operatorname{supp}(\varphi_n) &=& \sum_{a=1}^n
\E\bigl[\#\operatorname{supp}(\varphi_n)|\# A(Y_n)=a\bigr]P\bigl[
\#A(Y_n)=a\bigr]
\\
&=& \sum_{a=1}^n \frac{\#F-1}{\#F}aP
\bigl[\#A(Y_n)=a\bigr]= \frac{\#F-1}{\#F} \E a(Y_n).
\end{eqnarray*}

By Lemma \ref{expectact}, we conclude that
\[
H(Y_n) \geq C\frac{\#F-1}{\#F}\E a(Y_n) \geq C
\frac{\#F-1}{\#F} n^{\b
(n)-\varepsilon}.
\]
\upqed
\end{pf*}


%

Note that the proof for the upper bound remains valid for the group
$G(S,H) <\operatorname{Aut}(T_{\bar d})$, that is, when the group $F$ is trivial, but
the lower bound is true only with a nontrivial finite group $F$
(otherwise Fact \ref{supp} is obviously empty).
%
\begin{remark}\label{IT}
In information theory, the entropy is the ``average number of digits''
needed to describe some data. In this heuristic point of view, Theorem
\ref{MTE} is a corollary of Lemma \ref{expectact}.

Indeed, the activity $a(w)$ of a word $w$ is equivalent to the size of
the minimal tree $T(w)$ defined in Section \ref{mta}; recall $a(w)
\leq\#\partial T(w) \leq D a(w)$. Moreover, the element $\gamma=_\G
w$ in $\G$ is described by Figure \ref{fig1}, with $\s_u$ for $u$
nonactive vertices, $w_v$ for $v$ active vertices and $\varphi(x)$ at
active boundary points $x$. As each of them is described by $\leq C$
digits for some $C$ depending uniquely on the bound $D$ on valency, the
element $\gamma$ is described with $\leq C \#T(w) \approx a(w)$
digits. Moreover, since any function $\varphi\dvtx A_\partial(w)
\rightarrow F$ appears equally likely, one needs at least $a(w)$ digits
to describe $\gamma$.

Figure \ref{fig1} is, in this sense, the ``best description'' of the
element $\gamma$ in $\G$ (position in the Cayley graph) represented
by the word $w$ (path in the Cayley graph). The loss of information
from $w$ to $\gamma$ is somewhat described by the graph structure of
the ascendance forest $\AF(w)$ of Section \ref{ascfor}.
\end{remark}

\subsection{Precise entropy exponent and oscillation phenomena}
Theorem \ref{MTE} exhibits a large variety of behaviors for entropy of
random walk. In particular, it implies Theorem \ref{MT} for $\frac
{1}{2} \leq\a\leq\b<1$.
%
\begin{corollary}\label{exactentropy}
For any $\frac{1}{2} \leq\b<1$, there is a valency sequence $\bbar d$
such that the entropy exponent of the random walk $Y_n$ on a finitely
generated saturated directed group $\G(S_{d_0},\HF)<\operatorname{Aut}(ET_{\bar d})$ is
\[
h(\G,\m)=\b.
\]
\end{corollary}
\begin{pf}
Take $d \leq D$ such that $\b_2 \leq\b_d= \frac{\log d}{ \log p}
\leq\b\leq\frac{\log D}{ \log P}=\b_D<1$. There exist $\l\in
[0,1]$ such that $\b=\frac{\log(d^\l D^{1-\l})}{\log(p^\l P^{1-\l
})}$. Define the sequence $\bbar d$ by $d_i \in\{d,D\}$ for all $i$ and
$\#\{i \leq n | d_i=d\}=[\l n]$. Then the exponent sequence $\b_{\bar
d}(n) \rightarrow\b$ and the corollary follows from Theorem \ref{MTE}.
\end{pf}
%
\begin{remark}\label{entgro}
This corollary shows in particular that any saturated directed group of
a binary tree has entropy exponent $h(\G,\m)=\frac{1}{2}$. This is
the case of the groups $\G_\o=F \wr_{\partial T_2} G_\o$ for $G_\o
$ an Aleshin--Grigorchuk group, for which a great variety of growth
behaviors are known. For instance, $\G_{(012)^\infty}$ has growth
function $b_{(012)^\infty}(r) \approx e^{r^{\a_0}}$ for an explicit
$\a_0<1$ \cite{BE}, and for any $\a\in[\a_0,1]$ there is a group
$G_{\o(\a)}$ with growth such that $\lim\frac{\log\log b_{\o(\a
)}(r)}{\log r}=\a$ \cite{Bri3}. Considering the entropy, all these
growth behaviors collapse to a unique entropy exponent.
\end{remark}
%
\begin{corollary}\label{oscil}
For any $\frac{1}{2} \leq\a< \b<1$, there is a valency sequence
$\bbar d$ such that the lower and upper entropy exponents of the random
walk $Y_n$ on a finitely generated saturated directed group $\G
(S_{d_0},\HF)<\operatorname{Aut}(ET_{\bar d})$ are
\[
\underline{h}(\G,\m)=\a\quad\mbox{and}\quad\overline{h}(\G,\m)=\b.
\]
\end{corollary}\eject
\begin{pf}
By Theorem \ref{MTE}, it is sufficient to construct an appropriate
exponent sequence in order to prove the corollary. Take $d \leq D$ such
that $\b_d \leq\a< \b\leq\b_D$. Construct a sequence $\bbar d$
such that $d_i \in\{d,D\}$ for all $i$ according to the following rules.

Recall the definition (see Section \ref{defbeta}) of $\b(n)$ as the
exponent satisfying $n^{\b(n)}=d_0\cddots d_{k(n)}$ where $k(n)$ as the
unique integer such that $\frac{p_{k(n)}}{n}<p_0 \cddots p_{k(n)} \leq
\frac{1}{n}$. In order to ease the reading of the present proof, use
the shortcut notation $p_0 \cddots p_{k(n)} \approx\frac{1}{n}$. There
exists a constant $C$ depending\break uniquely on $D$ such that the following
statements are true if all relations $x \approx y$ below are replaced
by $\frac{y}{C} \leq x \leq Cy$.

Suppose $k(n)$ is such that $p_0 \cddots p_{k(n)} \approx\frac{1}{n}$
and $d_0\cddots d_{k(n)} \approx n^\a$ for some $n$ [hence $|\b(n)-\a
|\leq\frac{C}{\log n}$], and then set $d_{k(n)+1}=\cddots
=d_{k(n)+l}=D$ for $l$ such that for some minimal integer $m$,
\begin{eqnarray*}
\frac{P^l}{n} &\approx& p_0\cddots p_{k(n)}P^{l}=
p_0\cddots p_{k(n)+l} \approx\frac{1}{m},
\\
n^\a D^l &\approx& d_0\cddots
d_{k(n)}D^{l}= d_0\cddots d_{k(n)+l}
\approx m^\b.
\end{eqnarray*}
This forces $n^\a D^l\approx ( \frac{n}{P^l}  )^\b$,
hence $l=\frac{\b-\a}{\log(P^\b D)} \log n+o(\log n)$.

Suppose $l(m)$ is such that $p_0\cddots p_{l(m)}\approx\frac{1}{m}$ and
$d_0\cddots d_{l(m)} \approx m^\b$ for some $m$ [hence $|\b(m)-\b|\leq
\frac{C}{\log m}$], then set $d_{l(m)+1}=\cdots=d_{l(m)+k}=d$ for $k$
such that for some minimal integer $n$,
\begin{eqnarray*}
\frac{p^k}{m} &\approx& p_0\cddots p_{l(m)}p^{k}=
p_0\cddots p_{l(m)+k}\approx\frac{1}{n},
\\
m^\b d^k&\approx& d_0\cddots
d_{l(m)}d^{k}= d_0\cddots d_{l(m)+k}
\approx n^\a.
\end{eqnarray*}
This forces $m^\b d^k\approx ( \frac{m}{p^k}  )^\a$,
hence $k=\frac{\a-\b}{\log(p^\a d)} \log m+o(\log m)$.

This process produces a sequence $\bbar d$ such that $\a-\frac{C}{\log
n} \leq\b(n) \leq\b+\frac{C}{\log n}$ for all $n$. Moreover, $d_0
\cddots d_{k(n)}\approx n^\a$ for infinitely many $n$ and $d_0 \cddots
d_{l(m)}\approx m^\b$ for infinitely many $m$. Hence $\underline
{h}(\G,\m)=\a$, and $\overline{h}(\G,\m)=\b$.
\end{pf}
%
\begin{remark}
The above two corollaries are obtained by taking particular instances
of exponent functions $\b(n)$. Theorem \ref{MTE} provides a variety
of behaviors for entropy functions. For instance, similarly to Theorem
7.2 on growth functions in \cite{Gri1} and \cite{Bri3}, there exists
uncountable antichains of entropy functions with $\underline{h}(\G,\m
)=\a<\b=\overline{h}(\G,\m)$ for any given $\frac{1}{2}\leq\a<
\b<1$, as is easily proved by playing with exponent functions.
\end{remark}

\subsection{Frequency of oscillations} Corollary \ref{oscil} shows
that the entropy exponent of a random walk can take different values at
different scales. In order to study the difference between such scales,
given $\frac{1}{2} \leq\a\leq\b\leq1$ and a function $H(n)$,
introduce the following quantities, called, respectively, \textit{upper} and
\textit{lower pseudo period exponents} of the function $H(n)$ for $\a$
and $\b$:
\begin{eqnarray*}
u_H(\a,\b) &=& \inf\bigl\{\nu|\exists N_0,\forall n
\geq N_0\mbox{, if } H(n) \leq n^\a,
\\
& &\hspace*{16.5pt} \mbox{then } \exists n \leq m \leq n^\nu, H(m) \geq
m^\b\bigr\},
\\
l_H(\a,\b) &=& \inf\bigl\{\l|\exists N_0,\forall m
\geq N_0\mbox{, if }H(m) \geq m^\b,
\\
& &\hspace*{27.4pt} \mbox{then } \exists m \leq n \leq m^\l, H(n) \leq
n^\a\bigr\}.
\end{eqnarray*}
For $H=H_{\G,\m}$ entropy of a finitely supported measure $\m$ on a
group $\G$, write $u_H(\a,\b)=u_{\G,\m}(\a,\b)$ and $l_H(\a,\b
)=l_{\G,\m}(\a,\b)$.

Note that by playing with the sequence $\bbar d$ in the proof of
Corollary \ref{oscil}, one can easily produce examples of random walks
with arbitrarily large pseudo period exponents, that is, low frequency.
To study how high the frequency may be, that is, how small the pseudo
periods, introduce the functions
\begin{eqnarray*}
u(\a,\b)&=&\inf\bigl\{u_{\G,\m}(\a,\b)|\mbox{$\m$ is a finitely
supported symmetric}
\\
& &\hspace*{126pt} \mbox{measure on a group $\G$}\bigr\},
\\
l(\a,\b)&=&\inf\bigl\{l_{\G,\m}(\a,\b)|\mbox{$\m$ is a finitely
supported symmetric}
\\
& &\hspace*{123pt} \mbox{measure on a group $\G$}\bigr\}.
\end{eqnarray*}

The finitely generated saturated directed groups of Theorem \ref{MTE}
will provide the upper bounds in the following.
%
\begin{theorem}\label{freq}
For $\frac{1}{2} <\a\leq\b<1$,
\[
u(\a,\b)=\frac{\a-1}{\b-1}\quad \mbox{and}\quad
\frac{\b}{\a} \leq l(\a,\b)
\leq\frac{\b-{1/2}}{\a-{1/2}}.
\]
\end{theorem}
\begin{pf} The lower bounds follow from elementary properties of
entropy. For a submultiplicative function $H(n) \leq n^\a$ implies
$H(kn) \leq k n^\a$, so that if $H(kn) \geq(kn)^\b$, then $(kn)^\b
\leq kn^\a$, so $kn\geq n^{({1-\a})/({1-\b})}$ and $u_H(\a,\b)
\geq\frac{1-\a}{1-\b}$. For an increasing function $H(m) \geq m^\b
$ and $H(n) \leq n^\a$ imply $n \geq m^{{\b}/{\a}}$,
so $l_H(\a,\b) \geq\frac{\b}{\a}$.

Concerning the upper pseudo period exponents, the proof of Corollary
\ref{oscil} shows that given $n$, one can take $m\leq C\frac{n}{P^l}$
for $l\leq(\frac{\b-\a}{\log(P^\b D)}+\varepsilon)\log n$, where
$C$ depends only on $D$ and $\varepsilon>0$ is arbitrary, that is,
$m\leq n^\nu$ for
\[
\nu=1-\frac{(\b-\a)\log P}{\log(P^\b D)}-\varepsilon\log P
=1-\frac{\b-\a}{\b+{\log D}/{\log(
({D-1})/{D^2})}}+\varepsilon|
{\log P}|.
\]
As $D$ tends to infinity,
\[
\frac{\log D}{\log(({D-1})/{D^2})}=\frac{\log D}{\log(D-1)-2\log
D}=\frac{1}{{\log(D-1)}/{\log D}-2}\longrightarrow-1,
\]
which proves
\[
u(\a,\b) \leq1-\frac{\b-\a}{\b-1}=\frac{\a-1}{\b-1}.
\]

Similarly for the lower pseudo period exponent, the proof of Corollary
\ref{oscil} shows that given $m$, one can take $n\leq C\frac{m}{p^k}$
for $k=(\frac{\a-\b}{\log(p^\a d)}+\varepsilon)\log m$, that is,
$n\leq m^\l$ for
\[
\l=1-\frac{(\a-\b)\log p}{\log(p^\a d)}-\varepsilon\log p=1-\frac
{\a-\b}{\a+{\log d}/{\log(({d-1})/{d^2})}}+\varepsilon |
{\log p}|.
\]
Taking $d=2$, this proves
\[
l(\a,\b)\leq1+\frac{\b-\a}{\a-{1/2}}=\frac{\b-
{1/2}}{\a-{1/2}}.
\]
\upqed
\end{pf}
%
\begin{remark}
The value $u(\a,\b)=\frac{\a-1}{\b-1}$ is tightly related to
subadditivity of entropy, which implies in particular that $\overline
{h}(\G,\m) \leq1$ for any group $\G$ with finitely supported
measure $\m$. Also Theorem \ref{MTE} shows that the upper bound
$l(\a,\b)\leq\frac{\b-{1/2}}{\a-{1/2}}$ is optimal among
saturated directed groups with the measure $\m$ equidistributed on
$S_{d_0}\HF S_{d_0}$. It is unclear whether this bound is optimal in
general. It could be related to question \ref{qlb} on lower bound
$\underline{h}(\G,\m) \geq\frac{1}{2}$.
\end{remark}

\subsection{Drift of the random walk} Theorem \ref{MTE} provides
estimates on the drift $L_{\G,\m}(n)=\E\|Y_n\|$ of the random walk
$Y_n$ of step distribution $\m$ equidistributed on $S_{d_0}\HF S_{d_0}$,
where \mbox{$\|\cdot\|$} is the word norm for some (arbitrary) generating set.
%
\begin{corollary}\label{cordrift}
For any $\varepsilon>0$ and $n$ large enough,
\[
\b(n)-\varepsilon\leq\frac{\log L_{\G,\m}(n)}{\log n} \leq\frac
{1+\b(n)}{2}+\varepsilon.
\]
\end{corollary}
\begin{pf}
Lemmas 6 and 7 in \cite{Ers3} show that there are $c_1,c_2,c_3>0$ with
\[
c_1 H_{\G,\m}(n) -c_2 \leq L_{\G,\m}(n)
\leq c_3 \sqrt{n\bigl(H_{\G,\m
}(n)+\log n}\bigr).
\]
Combine with Theorem \ref{MTE}.
\end{pf}

\section{Return probability}

\subsection{General estimates}

Given a valency sequence $\bbar d=(d_i)_i$, set
\[
l(n)=l_{\bar d}(n)=\max\biggl\{l\Big| \frac{d_0^3}{d_0-1} \cddots
\frac
{d_l^3}{d_l-1}= \frac{d_0}{p_0} \cddots\frac{d_l}{p_l} \leq n\biggr\}\vadjust{\goodbreak}
\]
and define the auxiliary exponent sequence
\[
\b'(n)=\b'_{\bar d}(n)=\frac{\log(d_0 \cddots d_{l(n)})}{\log n}.
\]
In particular, $p_0\cddots p_{l(n)}n \approx d_0\cddots d_{l(n)}=n^{\b'(n)}$.
%
\begin{theorem}\label{RPE}
For any $\varepsilon>0$ the return probability of the simple random
walk $Y_n$ with generating set $S_{d_0}\HF S_{d_0}$ on the saturated
directed group $\G(S_{d_0},\HF)<\operatorname{Aut}(ET_{\bar d})$ satisfies for $n$
large enough,
\[
\b'(n) -\varepsilon\leq\frac{\log(-\log P(Y_n=id))}{\log n}\leq\b (n)+\varepsilon.
\]
In particular, $\underline{p}(\G)\geq\liminf\b'(n)$ and $\overline
{p}(\G) \leq\limsup\b(n)$.
\end{theorem}

For instance in the case of constant valency $d$, one has
\begin{eqnarray*}
\b'_d&=&\frac{1}{3-{\log(d-1)}/{\log d}} \leq\frac{\log(-\log
P(Y_n=id))}{\log n}\\
&\leq&\frac{1}{2-{\log(d-1)}/{\log d}}=\b_d.
\end{eqnarray*}
Note that $\b'_2=\frac{1}{3}$ and $\b'_d \longrightarrow_{d
\rightarrow
\infty} \frac{1}{2}$ for lower bounds, compared with $\b_2=\frac{1}{2}$
and $\b_d \longrightarrow_{d \rightarrow\infty} 1$ for upper bounds.
\begin{pf*}{Proof of Theorem \ref{RPE}}
Proposition \ref{expectact} implies that for $n$ large enough,
\[
P\bigl(a(Y_n) \geq n^{\b(n)+\varepsilon}\bigr)
\leq\frac{\E a(Y_n)}{n^{\b
(n)+\varepsilon}}
\leq\frac{n^{\b(n)+{\varepsilon
}/{2}}}{n^{\b(n)+\varepsilon}}=n^{-{\varepsilon
}/{2}}\longrightarrow0.
\]
Using the fact that for fixed $n$, the function $P(Y_n=\gamma)$ is
maximal at $\gamma=id$ by symmetry of the random walk, and the
inequality of Proposition \ref{counta}, deduce
\begin{eqnarray*}
P\bigl(a(Y_n)\leq n^{\b(n)+\varepsilon}\bigr)&=&\sum
_{\{
\gamma| \exists w=_\G\gamma, a(w) \leq n^{\b(n)+\varepsilon}\}} P(Y_n=\gamma)
\\
&\leq& P(Y_n=id) C^{n^{\b(n)+\varepsilon}}.
\end{eqnarray*}
As the left-hand term tends to $1$, this proves the upper bound.

Recall that given a word $Y_n$ of length $n$, the rewriting process
provides for each vertex $v \in T_{\bar d}$ a word $Y_n^v$ of random
length $m_v$ (Proposition \ref{rp}). Given $\th>0$ small enough so
that for any $n$ large enough $l(n) \leq k^{-\th}(n)$ [defined in
Section \ref{defbeta} by $(p_0-\th)\cddots(p_{k^{-\th}(n)}-\th
)n\approx N_0$ for an arbitrary $N_0\geq1$, so that the inequality
holds for large $n$ as soon as $\frac{d_i-1}{d_i^2}-\th> \frac
{d_i-1}{d_i^3}$ for all $i$], observe the following:
%
\begin{fact}\label{vercon}
For a word $Y_n$ with $a(Y_n) < n^{\b'(n)}$, there exists a vertex $v$
in $T_{\bar d}$ such that:
\begin{longlist}[(3)]
\item[(1)] $|v| \leq l(n)$;
\item[(2)] there is $t$ such that $m_{vt} < (p_{|v|}-\th)m_v$;
\item[(3)] $m_v \geq(p_0-\th) \cddots(p_{|v|-1}-\th)n$.
\end{longlist}
\end{fact}
\begin{pf}
By contradiction assume that for all $|v| \leq l(n)$ and $t$, one has
$m_{vt} \geq(p_{|v|}-\th)m_v$. Then by induction on $|v|$, for all
$v$ in level $l(n)$,
\[
m_v \geq(p_0-\th)\cddots(p_{l(n)-1}-\th)n \geq
\frac
{N_0}{C(p_{l(n)}-\th)\cddots(p_{k^{-\th}(n)}-\th)}\geq1,
\]
so that $a(Y_n^v) \geq1$. Then $a(Y_n) \geq\sum_{|v|=l(n)}a(Y_n^v)
\geq d_0\cddots d_{l(n)}=n^{\b'(n)}$, which is a contradiction. This
shows the existence of a vertex $v$ satisfying (1) and (2). Such a
vertex that is closest to the root also satisfies (3).
\end{pf}

Fact \ref{vercon} guarantees that $P(a(Y_n) \leq n^{\b'(n)})$ is
bounded above by
\[
\sum_{|v| \leq l(n)} P\bigl[\exists t, m_{vt}
\leq(p_{|v|}-\th) m_v \mbox{ and } m_v
\geq(p_0-\th) \cddots(p_{|v|-1}-\th)n\bigr].
\]
Now $P[m_{vt} \leq(p_{|v|}-\th) m_v \mbox{ and } m_v]=P[m_{vt}
\leq(p_{|v|}-\th) m_v |m_v]P(m_v)$, and by Proposition \ref
{randomRP} there is $c_\th<1$ with $ P[m_{vt} \leq(p_{|v|}-\th) m_v
|m_v] \leq c_\th^{m_v}$. Then $P(a(Y_n) \leq n^{\b'(n)})$ is bounded
above by
%
\begin{equation}
\label{sar} \sum_{|v|\leq l(n)} c_\th^{(p_0-\th)\cddots
(p_{l(n)-1}-\th)n}=n^{\b'(n)}c_\th^{(p_0-\th)\cddots(p_{l(n)-1}-\th
)n},
\end{equation}
because there are $d_0\cddots d_{l(n)}= n^{\b'(n)}$ vertices such that
$|v| \leq l(n)$. Compute by conditioning on activity, recalling that
Fact \ref{cond} ensures $P[\varphi_n=id|a(Y_n)=a]=(\frac{1}{\#F})^a$.
\begin{eqnarray*}
P(Y_n=id) &\leq& P(\varphi_n =id) = \sum
_{a=0}^n P\bigl[\varphi_n=id|a(Y_n)=a
\bigr]P\bigl(a(Y_n)=a\bigr),
\\
&=& \sum_{a=0}^n \biggl(
\frac{1}{\#F} \biggr)^a P\bigl(a(Y_n)=a\bigr).
\end{eqnarray*}
The decay of $ ( \frac{1}{\#F}  )^a$ with $a$ allows us to
split the sum between $a < n^{\b'(n)}$ and $a \geq n^{\b'(n)}$. Obtain
\begin{eqnarray*}
P(Y_n=id) &\leq& P\bigl(a(Y_n) < n^{\b'(n)}\bigr)+
\biggl(\frac{1}{\#F} \biggr)^{n^{\b'(n)}}P\bigl(a(Y_n) \geq
n^{\b'(n)}\bigr),
\\
&\leq& n^{\b'(n)}c_\th^{(p_0-\th)\cddots(p_{l(n)-1}-\th)n}+ \biggl(
\frac{1}{\#F} \biggr)^{n^{\b'(n)}}
\end{eqnarray*}
by inequality (\ref{sar}). Now there is a function $\varepsilon(\th)
\longrightarrow_\th0$ such that $(p_0-\th)\cddots(p_{l(n)-1}-\th)n
\geq n^{\b'(n)-\varepsilon(\th)}$ (cf. Proposition \ref
{thetapprox}), so
\[
P(Y_n=id) \leq\exp\bigl(-c n^{\b'(n)-\varepsilon(\th)}\bigr),
\]
which proves the lower bound.
\end{pf*}

As for Theorem \ref{MTE}, the upper bound is valid for $G(S,H)$, that
is, for $F=\{1\}$, but the lower bound is valid only with a nontrivial
finite group $F$.

\begin{remark}
Fact \ref{vercon} shows that if the activity is small, then there is
at least one edge along which the word length is contracted more than
expected, with $m_{vt} \leq(p_{|v|}-\th)m_v$ instead of $m_{vt} =
p_{|v|}m_v$. In fact, in order to have $a(Y_n)<n^{\b'(n)}$, such a
strong contraction must occur at many edges, so that inequality (\ref
{sar}) does not seem optimal. The example below shows that the lower
bound of Theorem~\ref{RPE} is tight, and thus inequality (\ref{sar})
is essentially optimal in general. It might, however, be improved for
particular instances of saturated directed groups.
\end{remark}

\subsection{A specific example} Consider the particular case of a
binary tree and specific generators $s=(1,1)\s$ with $\s$ the
nontrivial permutation in $S_2$ and $h=(h,s)$. The group $\langle
s,h\rangle<\operatorname{Aut}(T_2)$ is a well-known automata group isomorphic to the
infinite dihedral group $D_\infty= \langle s,h|s^2=h^2=1\rangle$.
%
\begin{proposition}\label{specific}
A random walk $Z_n$ on the extended directed group $\G(S,\HF)=F \wr_{\partial T_2} D_\infty<\operatorname{Aut}(ET_2)$ with $F$ Abelian finite with step
distribution equidistributed on $S\HF HS$ satisfies
\[
\lim\frac{{\log}|{\log P}(Z_n=id)|}{\log n}=\frac{1}{3}
\quad\mbox{and}\quad \lim
\frac{\log\E\|Z_n\|}{\log n}=\frac{1}{2}.
\]
\end{proposition}

In this particular case, the lower bounds of Theorem \ref{RPE} and
Corollary \ref{cordrift} are tight. Note that this specific group
played a crucial role in the construction of antichains of growth
functions in \cite{Gri1} and in the construction of groups with
oscillating growth in \cite{Bri1} (Chapter 2).
\begin{pf*}{Proof of Proposition \ref{specific}}
Theorem \ref{RPE} and Corollary \ref{cordrift} ensure
\[
\liminf\frac{{\log}|{\log P}(Z_n=id)|}{\log n}\geq\frac{1}{3}
\quad\mbox{and}\quad \lim
\frac{\log\E\|Z_n\|}{\log n}\geq\frac{1}{2}.
\]
To get an upper bound, compare with the usual wreath product $F \wr_{D_\infty} D_\infty$, for which the return probability satisfies
$P(\tilde{X}_n=id) \approx e^{-n^{1/3}}$ (Theorem 3.5 in \cite
{PSC2}, noting that $\Z$ and $D_\infty$ with their usual generating
sets have the same Cayley graph) and the drift $L_{F\wr_{D_\infty
}D_\infty}(n)\approx n^{{1}/{2}}$ (by Lemma 3 in
\cite{Ersdrift}).\vadjust{\goodbreak}

Precisely, consider the usual wreath product (lamplighter)
\[
F \wr_{D_\infty} D_\infty= \bigl\{\F\dvtx D_\infty\rightarrow
F|{\operatorname{supp}}(\F) \mbox{ is finite}\bigr\} \rtimes D_\infty
\]
with the action $d\doot\F(x)=\F(xd)$. Denote $S=\langle s\rangle\simeq
S_2$ and $H=\langle h \rangle\simeq S_2$, so that $D_\infty\simeq
G(S,H) <\operatorname{Aut}(T_2)$, and denote $F$ the subgroup $\{\F\dvtx D_\infty
\rightarrow F|\forall x \neq1_{D_\infty},\F(x)=1_F \}$ in $F \wr_{D_\infty} D_\infty$. Let $\tilde{X}_n$ be the random walk with
alternate successive increments equidistributed in the finite symmetric
sets $HSH$ and $F$ (this is, up to negligible first and last factors,
the random walk with step distribution $\m$ equidistributed on the
finite symmetric set $S\HF HS$)
\[
\tilde{X}_n=h_1s_1h'_1
f_1 h_2s_2h'_2
f_2\cddots h_ns_nh'_nf_n.
\]
It induces in particular a random walk $X_n$ on the base group
$D_\infty$ given by
\[
X_n=h_1s_1h'_1
h_2s_2h'_2\cddots
h_ns_nh'_n=r_1r_2
\cddots r_n
\]
with $r_i=h_is_ih'_i$. The value of $\tilde{X}_n$ in $F \wr_{D_\infty
} D_\infty$ is given by the value of $X_n$ in $D_\infty$
together\vspace*{2pt}
with a function $\F_n\dvtx D_\infty\rightarrow F$ the support of which is
included in $\{r_1^{-1},(r_1r_2)^{-1},\ldots,(r_1r_2\cddots r_n)^{-1}\}$,
since at time $i$ the lamp in position $(r_1r_2\cddots r_i)^{-1}$ is modified.

With obvious identification of $S,H,F$, denote $Z_n$ the similar random
walk on $F \wr_{\partial T_2} G(S,H)$,
\[
Z_n=h_1s_1h'_1
f_1 h_2s_2h'_2
f_2\cddots h_ns_nh'_nf_n.
\]
The value of $Z_n$ is given by the value of $X_n$ in $G(S,H) \simeq
D_\infty$ and a function $\varphi_n\dvtx\partial T_2 \rightarrow F$ with
support included in the active boundary set $\{1^\infty
r_1^{-1},\allowbreak1^\infty(r_1r_2)^{-1},\ldots,1^\infty(r_1r_2\cddots r_n)^{-1}\}
$ by Proposition \ref{counta}.

Now the Schreier graph $1^\infty G(S,H)$ can easily be described as a
semi-line. If $w$ is a reduced representative word in $D_\infty
=\langle s,h|s^2=h^2=1\rangle$, then $1^\infty hw=1^\infty w$ and
$1^\infty w \neq1^\infty w'$ if $w \neq w'$ and both $w$ and $w'$
start with $s$, so there is a canonical 2-covering application
$c\dvtx \operatorname{Cay}(D_\infty,\{s,h\})\twoheadrightarrow1^\infty G(S,H)$ with
$c(w)=c(hw)$. This implies
\[
\varphi_n(x)=\sum_{y \in c^{-1}(x)}
\F_n(y),
\]
because $F$ is Abelian, so that the order of increments does not
influence the sum. This shows that if $\tilde{X}_n= 1$ in $F \wr_{D_\infty}D_\infty$ then $Z_n=1$ in $F \wr_{\partial T_2} G(S,H)$,
hence $P(Z_n=1) \geq P(\tilde{X}_n=1)$ and the group $F \wr_{\partial
T_2} G(S,H)$ is a quotient of $F \wr_{D_\infty}D_\infty$ with
identification of the generators, so $\|\tilde{X}_n\| \leq\|Z_n\|$.
\end{pf*}

\section{Higher order oscillations}\label{highorder}

This section aims at proving Theorem \ref{MTT} and treating the case
$\b=1$ in Theorem \ref{MT}. The following construction is designed to
obtain groups $\D(S',H'F')$ that resemble\vadjust{\goodbreak} $\G(S,\HF)$ at some scales
and nonamenable groups at other scales. They are still directed groups
of a rooted tree but of unbounded valency $\bbar d$ (in the cases of
interest here) and not saturated.

The construction generalizes \cite{KP}, where a group $\D(S',H'_\o)$
is constructed given an Aleshin--Grigorchuk group $\G(S,H_\o)$ and
additional data. Theorem \ref{RPE} and Corollary \ref{cordrift} allow
us to show that some of the groups $\D=\D(S',H'_\o)$ or almost
equivalently some of the piecewise automatic groups of \cite{Ers2} satisfy
\[
\underline{p}(\D)\leq\tfrac{1}{2},\qquad \overline{p}(\D)=1
\quad\mbox{and}\quad
\underline{\delta}(\D)\leq\tfrac{3}{4},\qquad \overline{\delta }(\D)=1.
\]

The description of the construction in \cite{KP} is more algebraic,
whereas the point of view adopted here is in terms of automorphisms of
an ambiant tree~$T_{\bar e}$.

\subsection{\texorpdfstring{Definition of $\D(S',H'F')$}{Definition of Delta(S',H'F')}} Given a bounded sequence
$\bbar d$ and a saturated directed finitely generated group $\G=\G
(S,\HF)<\operatorname{Aut}(ET_{\bar d})$, construct a modification $\D$ of this group,
acting on an extended spherically homogeneous rooted tree $ET_{\bar e}$
for another sequence $\bbar e=(e_l)_{l \in\N}$, with $e_l=d_l+d'_l$
for some $d'_l \geq0$, possibly $d'_l=\infty$. Note that there is a
canonical inclusion $ET_{\bar d} \subset ET_{\bar e}$, and hence a
canonical embedding,
\[
\operatorname{Aut}(ET_{\bar d}) \hookrightarrow \operatorname{Aut}(ET_{\bar e}).
\]

Corresponding to the group $\G(S,\HF)$, determined by the finite groups
$S,H,F$ and the portraits of their elements, that is, their realization
in $\operatorname{Aut}(ET_{\bar d})$, construct a new group $\D(S',H'F') <
\operatorname{Aut}(ET_{\bar e})$, where $S \simeq S'$, $H \simeq H'$ and $F \simeq F'$
as abstract groups. Define the generators via the wreath product
isomorphism of Proposition \ref{CEpourET}.

\begin{longlist}[(4)]
\item[(1)] The element $s'$ in $S'$ corresponding to $s \in S \simeq
S'$ has the form
\[
s'=(1,\ldots,1)s',
\]
where $s'$ is a permutation of $\{1,\ldots,d_0,d_0+1,\ldots,d_0+d'_0\}$
respecting the decomposition $\{1,\ldots,d_0\}
\sqcup\{d_0+1,\ldots,d_0+d'_0\}$, so that $s'|_{\{1,\ldots,d_0\}}=s \in
S=S_{d_0} <
S_{e_0}$ (canonical inclusion) and moreover $s''=s's^{-1}=s'|_{\{
d_0+1,\ldots,d_0+d'_0\}}$ is chosen so that $s \mapsto s''$ is a
morphism of groups from $S \simeq S'$ into the permutation group $S_{\{
d_0+1,\ldots,d_0+d'_0\}}\simeq S_{d'_0}$. Denote by $S''$ its image.

\item[(2)] The element $h'$ in $H'$ corresponding to $h=(h_1,\s_2,\ldots,\s_{d_0}) \in H \simeq H'$ has the form
\[
h'=\bigl(h_1',\s'_2,
\ldots,\s'_{d_0},1,\ldots,1\bigr)h'',
\]
where $\s'_2,\ldots,\s'_{d_0} \in S'_1$ are defined as above,
corresponding, respectively, to $\s_2,\ldots,\s_{d_0} \in S_1$, $h_1'
\in H'_1$ corresponds to $h_1 \in H_1$, and $h''$ is a permutation in
$S_{d_0+d_0'}$ with support included in $\{d_0+1,\ldots,d_0+d'_0\}$ so
that $h \mapsto h''$ is a morphism of groups with image $H'' < S_{d_0'}$.

Elements of $H'$ do not act at the boundary $\partial T_{\bar e}$.

\item[(3)] The element $f'$ in $F'$ corresponding to $\varphi_f \in F
\simeq F'$ has the form
\[
f'=\bigl(f'_1,1,\ldots,1
\bigr)f'',
\]
where $f'_1 \in F'$ corresponds to $\varphi_f \in\G(S_1,H_1F)$, and
$f''$ is a permutation in $S_{d_0+d_0'}$ with support included in $\{
d_0+1,\ldots,d_0+d'_0\}$ so that $f \mapsto f''$ is a morphism of
groups with image $F''<S_{d_0'}$.

The element $f'$ acts at the boundary $\partial T_{\bar e}$ by $\varphi
(x)=1_F$ if $x \neq1^\infty$ and \mbox{$\varphi(1^\infty)=f$}.

\item[(4)] Any two elements $f''$ and $h''$ in
$S_{\{d_0+1,\ldots,d_0+d'_0\}}$ commute.
\end{longlist}

Note that similarly to the definition of $h \in \operatorname{Aut}(T_{\bar d})$ in
Section \ref{directedgroups}, the definitions of $h'$ and $f'$ are
recursive for they involve the generators $h_1'$ and $f_1'$ of the
group $\D(S_1',H_1'F_1')<\operatorname{Aut}(ET_{\s\bar e})$ associated to the
saturated directed group $\G(S_1,H_1F)<\operatorname{Aut}(ET_{\bar e})$.

Condition (4) implies by recursion that at any level $l$ the elements
$h''_l$ and $f''_l$ in $S_{\{d_l+1,\ldots,d_l+d'_l\}}$ commute, so the
subgroup $\langle\{s_l''\},\{h_l''\},\{f_l''\} \rangle< S_{d'_l}$ is
a quotient of the free product of finite groups $S'_l \ast(H_l' \times
F_l') \simeq S_l \ast(H_l \times F)$, possibly an infinite quotient in
the case $d_l'= \infty$. Denote $S_l''=\{s_l''\},H_l''=\{h_l''\}
,F_l''=\{f_l''\}$. They are subgroups of $S_{\{d_l+1,\ldots,d_l+d'_l\}}$.

By recursion, the action of the generators $s',h',f'$ is well defined
on the whole tree $T_{\bar e}$. Moreover, only generators $f'$ act
nontrivially at the boundary $\partial T_{\bar e}$; thus the action of
$s',h',f'$ on $ET_{\bar e}$ is well defined.

To summarize, the group $\D(S',H'F')$ is defined by $\G(S,\HF)$ and a
collection of groups $\langle S''_l,H''_lF''_l\rangle< S_{d'_l}$ that
are quotients of the free products of finite groups $S_l \ast H_lF$ by
identification of generators; see Figure~\ref{fig}.
%
\begin{figure}[b]

\includegraphics{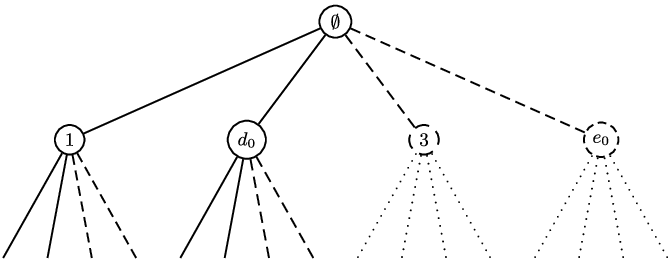}

\caption{The tree $T_{\bar e}$, for $d_0=d_1=2$ and $e_0=e_1=4$, with
the subtree $T_{\bar d}$ in plain edges, dashed edges where $\langle
S''_l,H''_lF''_l \rangle<S_{d'_l}$ act, and dotted edges where the
action of $\D$ is trivial.}
\label{fig}
\end{figure}

\subsection{\texorpdfstring{Elementary properties of $\D(S',H'F')$}{Elementary properties of Delta(S',H'F')}}
Note that $\D (S',H'F')$ is directed but not saturated. As a shortcut, write $\G_l=\G(S_l,H_l F) <
\operatorname{Aut}(ET_{\s^l \bar d})$ and $\D_l=\D (S'_l,H'_lF'_l)<
\operatorname{Aut}(ET_{\s^l \bar e})$.
%
\begin{properties}\label{deltaproperties}
\textup{(1)} The canonical isomorphism of Proposition \ref{CEpourET}
induces a canonical embedding
\[
\D\bigl(S',H',F'\bigr) \hookrightarrow
\D\bigl(S'_1,H'_1,F'_1
\bigr) \wr\bigl\langle S',H'',F''
\bigr\rangle.
\]
More generally, $\D_l \hookrightarrow\D_{l+1} \wr\langle
S_l',H_l'',F_l'' \rangle$ for any $l$.

\textup{(2)} Any two elements $h' \in H'$ and $f' \in F'$ commute.

\textup{(3)} The group $\G(S,H,F)<\operatorname{Aut}(ET_{\bar d})$ is a quotient of $\D
(S',H',F')<\operatorname{Aut}(ET_{\bar e})$.
\end{properties}

Property (2) shows that $\langle H',F' \rangle\simeq H' \times F'$.
Thus as canonical generating set of $\D(S',H'F')$, use $S' \sqcup
H'F'$.
\begin{pf*}{Proof of Properties \ref{deltaproperties}}
(1) is obvious by construction. For (2), compute (recall the support of
$f''$ and $h''$ is in $\{d_0+1,\ldots,d_0+d_0'\}$)
\[
\bigl[h',f'\bigr]=\bigl(\bigl[h_1',f_1'
\bigr],\bigl[\s_2',1\bigr],\ldots,\bigl[
\s_{d_0}',1\bigr],1,\ldots,1\bigr) \bigl[h'',f''
\bigr]=\bigl(\bigl[h_1',f_1'
\bigr],1,\ldots,1\bigr).
\]
By recursion, this shows that $[h',f']$ has a trivial action on
$T_{\bar e}$. Moreover, the support of the corresponding function
$\partial T_{\bar e} \rightarrow F$ is contained in $1^\infty$ where
the value is $[1,f]=1_F$. This shows $[h',f']=_\D1$.

For (3), observe that all the permutations involved in the description
of $S',H',F'$ respect the decomposition $\{1,\ldots,d_l\}\sqcup\{
d_l+1,\ldots,d_l+d_l'\}$ for all $l \in\N$, so that the subset
$ET_{\bar d} \subset ET_{\bar e}$ is stable under the action of $\D
(S',H',F')$. The quotient action is precisely given by $\G(S,\HF)$.
\end{pf*}

Observe that if $d_l=0$ for all $l \geq1$, then $\D(S',H'F')\simeq\G
(S,\HF) \times\langle S'',\allowbreak H''F'' \rangle$. In particular, if $d_l'=0$
for all $l \in\N$, then $\D(S',H'F')=\G(S,\HF)$.

A word $w(S,\HF)$ is an ordered product of generators $s$ in $S$, $h$ in
$H$ and $\varphi_f$ in $F$ of the group $\G$. By replacing $s$ by
$s'$, $h$ by $h'$ and $\varphi_f$ by $f'$, one naturally obtains a
word $w(S',H'F')$ in the generators of $\D$. The rewriting process of
Proposition~\ref{rp} applies to the groups $\D(S',H'F')$. More precisely:

\begin{proposition}[{[Rewriting process for groups $\D(S',H'F')$]}]\label{rp2}
If the rewriting process in the group $\G\hookrightarrow\G_1 \wr
S_{d_0}$ provides
\[
w(S,\HF)=\bigl(w^1(S_1,H_1F),
\ldots,w^{d_0}(S_1,H_1F)\bigr)\s_\varnothing,
\]
then the rewriting process in the group $\D\hookrightarrow\D_1 \wr
S'$ provides
\[
w\bigl(S',H'F'\bigr)=
\bigl(w^1\bigl(S_1',H_1'F_1'
\bigr),\ldots,w^{d_0}\bigl(S_1',H_1'F_1'
\bigr),1,\ldots,1\bigr) w\bigl(S',H''F''
\bigr),
\]
where $w(S',H''F'')=\s_\varnothing w(S'',H''F'')$ with $\s_\varnothing
\in S_{\{1,\ldots,d_0\}}$, and $w(S'',H''F'')$ takes values in $S_{\{
d_0+1,\ldots,d_0+d_0'\}}$.
\end{proposition}
\begin{pf}
Denote $w(S,\HF)=s_1h_1 \varphi_{f_1}s_2\cddots h_n\varphi_{f_n}
s_{n+1}$. Then
\[
w\bigl(S',H'F'\bigr)=s'_1h'_1f'_1s'_2,
\ldots, h'_nf'_ns'_{n+1},
\]
where $s \mapsto s'$, $h \mapsto h'$ and $\varphi_f \mapsto f'$ are
given in the definition of generators of $\D$. These forms are
equivalent to
\begin{eqnarray*}
w(S,\HF)&=&(h_1 \varphi_{f_1})^{\s_1}
\cddots(h_n \varphi_{f_n})^{\s_n}\s_{n+1}
\qquad\mbox{for } \s_i=s_1\cddots s_i,
\\
w\bigl(S',H'F'\bigr)&=&
\bigl(h'_1 f'_1
\bigr)^{\s'_1}\cddots\bigl(h_n f'_n
\bigr)^{\s'_n}\s'_{n+1}\qquad \mbox{for }
\s'_i=s'_1\cddots
s'_i.
\end{eqnarray*}

For $h \varphi_f=(h_1\varphi_{f},b_2,\ldots,b_{d_0})$, one has
\[
(h \varphi_f)^\s=(b_{\s^{-1}(1)},
\ldots,h_1\varphi_{f},\ldots,b_{\s
^{-1}(d_0)})
\]
with $h_1\varphi_{f}$ in position $\s(1)$.

Correspondingly, for $h'f'=(h'_1f'_1,b'_2,\ldots,b'_{d_0},1,\ldots,1)
h''f''$, one has
\[
\bigl(h'f'\bigr)^{\s'}=
\bigl(b'_{\s^{-1}(1)},\ldots,h'_1f'_1,
\ldots,b'_{\s
^{-1}(d_0)},1,\ldots,1\bigr) \bigl(h''f''
\bigr)^{\s''}
\]
with $h'_1f'_1$ in position $\s^{-1}(1)=\s'^{-1}(1)$, because $\s'=\s\s''$ with $\s\in S_{\{1,\ldots,d_0\}}$ and $\s'' \in S_{\{
d_0+1,\ldots,d_0+d'_0\}}$.

As $(h''f'')^{\s''}$ acts trivially on $\{1,\ldots,d_0\}$, one can
compute products (i.e., words) componentwise, which proves the proposition.
\end{pf}
%
\begin{remark}
Proposition \ref{rp2} shows in particular that the minimal tree
description of Section \ref{mta} and Figure \ref{fig1} for a word
$w(S,\HF)$ in $\G(S,\HF)$ remains valid for the word $w(S',H'F')$ in $\D
(S',H'F')$ with the same tree $T(w)<T_{\bar d}<T_{\bar e}$, but for a
vertex $u$ in level $l$, the permutation $\s_u$ now takes values in
$S_l \times\langle S''_l,H''_lF''_l\rangle<S_{e_l}$.
\end{remark}

\subsection{Localization} In the Cayley graph of $\D(S',H'F')$, a
ball of given radius depends only on the description of the action on a
neighborhood of the root of the tree.
%
\begin{proposition}\label{local}
The ball $B_\D(R)$ in the Cayley graph of $\D(S',H'F')$ with respect
to the generating set, $S' \sqcup H'F'$ depends only on the $L=1+\log_2 R$ first levels in the recursive description of the generators of
$\D(S',H'F')$ and the $2R+1$-balls in the groups $\langle
S''_l,H''_lF''_l \rangle< S_{d_l'}$ for $l \leq L$.
\end{proposition}
\begin{pf}
The ball $B_\D(R)$ can be drawn if one can test (algorithmically) the
oracle $w=_\D1$ for any given word $w$ in $S' \sqcup H'F'$ of length
$\leq r=2R+1$. To test such an oracle, use the following algorithm.

First test the value of the permutation induced by $w$ at the root,
given by $\F(w(s',h'f'))=w(s',h''f'')$, where $\F\dvtx \operatorname{Aut}(ET_{\bar e})
\rightarrow S_{e_0}$ is the root evaluation. This test depends only on
the zero level in $\D(S',H'F')$ and the $|w|$-ball in $\langle
S''_0,H''_0F''_0 \rangle< S_{d_0'}$.\vspace*{2pt}

If $\F(w) \neq_{S_{e_0}} 1$, then $w \neq_{\D} 1$. If $\F(w)
=_{S_{e_0}} 1$, then $w$ fixes all vertices in the first level of
$T_{\bar e}$. By Propositions \ref{rp} and \ref{rp2} of the rewriting
process, one\vspace*{1pt} has $w=(w^1,\ldots,w^{d_0},1,\ldots,1)$ in the wreath
product with $|w^t| \leq\frac{|w|+1}{2} \leq\frac{r+1}{2}$.

Test for each $t$ in $\{1,\ldots,e_0\}$, the permutation induced by
$w^t$ at the root $\F(w^t(s_1',h_1'f_1'))=w^t(s_1',h_1''f_1'')$, which
depends only on the zero level in $\D(S'_1,H'_1F'_1)$ hence on the
first level in $\D(S',H'F')$ and the $|w^t|$-ball in $\langle
S''_1,\allowbreak H''_1F''_1 \rangle< S_{d_1'}$.

If $\F(w^t) \neq_{S_{e_1}} 1$, then $w \neq_{\D} 1$. If $\F(w^t)
=_{S_{e_0}} 1$ for all $t$, then $w$ fixes all vertices in the two
first levels of $T_{\bar e}$. By Propositions \ref{rp} and \ref{rp2}
of the rewriting process, $w^t=(w^{t1},\ldots,w^{te_0},1,\ldots,1)$ in
the wreath product with $|w^{ts}| \leq\frac{|w^t|+1}{2} < \frac{r}{4}+1$.

Continue the process and test the value at the root of the words
$w^{t_1\cddots t_l}$ while their length is $\geq1$. This test depends
only on the $l$ first levels in $\D(S',H'F')$ and the $|w^{t_1\cddots
t_l}|$-ball in $\langle S''_l,H''_lF''_l \rangle< S_{d_l'}$. If $\F
(w^{t_1\cddots t_l}) \neq_{S_{e_l}} 1$ for some $t_1\cddots t_l$, then $w
\neq_\D1$.

For $L=\log_2 r$, one has $|w^{t_1\cddots t_L}|< \frac{r}{2^L}+ 1=2$,
so $w^{t_1\cddots t_L}$ is a generator in $\D_L(S_L',H_L'F_L')$. This
implies that if $\F(w^{t_1\cddots t_l})=_{S_{e_l}} 1$ for all $t_1\cddots
t_l$, $l \leq L-1$ and $w^{t_1\cddots t_L}=_{\D_L}1$, then $w=_\D1$.

The algorithm allows us to test the oracle $w=_\D1$ using only the
data in the $L=\log_2 R$ first levels of $\D(S',H'F')$ and the $\frac
{r}{2^l}$-ball in the group $\langle S''_l,H''_lF''_l \rangle<
S_{d_l'}$ for $l \leq\log_2 r$.
\end{pf}

\subsection{Asymptotic properties} The asymptotic description of $\D
(S',H'F')$ is well understood in two extreme cases.
%
\begin{proposition}[(Low asymptotic)]\label{LA}
If $d'_l=0$ for $l \geq L+1$ and $d'_l$ finite for $l \leq L$, the
quotient homomorphism (of restriction of the action to $ET_{\bar d}
\subset ET_{\bar e}$)
\begin{eqnarray*}
f\dvtx\D\bigl(S',H'F'\bigr)
&\rightarrow& \G(S,\HF),
\\
s' & \mapsto& s,
\\
h' & \mapsto& h,
\\
f' & \mapsto& \varphi_f
\end{eqnarray*}
has finite kernel.

In particular, for the random walks $Y_n$ in $\G(S,\HF)$ of law $\m$
equidistributed on $S\HF S$ and the associated random walk $Y'_n$ in $\D
(S',H'F')$ of law $\m'$ equidistributed on $S'H'F'S'$, there exists
$C,K >0$ such that for all~$n$:
\begin{longlist}[(3)]
\item[(1)] $P(Y_n'=_\D1) \leq P(Y_n =_\G1) \leq C P(Y'_n=_\D1)$;
\item[(2)] $L_{\G,\m}(n) \leq L_{\D,\m'}(n) \leq L_{\G,\m}(n)+K$;
\item[(3)] $H_{\G,\m}(n) \leq H_{\D,\m'}(n) \leq H_{\G,\m
}(n)+\log C $.
\end{longlist}
\end{proposition}
\begin{pf}
An element $\delta$ in the kernel $\ker f=\{\delta\in\D| \delta
|_{ET_{\bar d}}=1\}$ is described by its action on $ET_{\bar e}
\setminus ET_{\bar d}$. By the rewriting process, one can write
$\delta=(\delta_1,\ldots,\delta_{d_0},1,\ldots,1)\delta''$ with
$\delta'' \in S_{d_0+1,\ldots,d_0+d_0'}<S_{e_0}$, for which there are
\mbox{$\leq$}$\#\langle S'',H''F'' \rangle$ choices. In order to describe
$\delta$, t we describe $\delta_1,\ldots,\delta_{d_0}$ that belong
to the kernel $\ker(f_1\dvtx\D_1 \rightarrow\G_1)$.

For each $t \in\{1,\ldots, d_0\}$, the element $\delta_t$ can be
written in the form $\delta_t=
(\delta_{t1},\ldots, \delta_{td_1},1,\ldots,1)\delta_t''$ with $\delta_t'' \in S_{d_1+1,\ldots,
d_1+d_1'}<S_{e_1}$, for which there are \mbox{$\leq$}$\#\langle
S_1'',H_1''F_1'' \rangle$ choices.

By induction, the element $\delta\in\ker f$ is described by
\[
\bigl\{\delta_{t_1\cddots t_l}''|t_i
\in\{1,\ldots,d_i\}, l \leq L\bigr\},
\]
for which the number of possible choices is
\[
C \leq\#\bigl\langle S'',H''F''
\bigr\rangle\bigl(\#\bigl\langle S_1'',H_1''F_1''
\bigr\rangle \bigr)^{d_0} \cddots\bigl(\#\bigl\langle S_L'',H_L''F_L''
\bigr\rangle\bigr)^{d_0\cddots d_{L-1}}.
\]

Denote $f^{-1}(1_\G)=\{\delta_1,\ldots,\delta_C\}$, then
$f^{-1}(\gamma)=\{\delta\delta_1,\ldots,\delta\delta_C\}$ if
$f(\gamma)=\delta$ and
\[
\m^{\ast n}(\gamma)=P(Y_n=_\G\gamma)=\sum
_{i=1}^CP\bigl(Y_n'=_\D
\delta \delta_i\bigr)=\sum_{\delta' \in f^{-1}(\gamma)}
\m'^{\ast n}\bigl(\delta'\bigr).
\]
For $\gamma=1$, this guarantees (1)
\[
P\bigl(Y_n'=_\D1\bigr) \leq
P(Y_n =_\G1) \leq C P\bigl(Y'_n=_\D1
\bigr),
\]
because $P(Y_n=_\D\delta)$ is maximal for $\delta=1_\D$.

One also has $\|\delta\| \leq\|f(\delta)\|+K$ for
$K=\max\{\|\delta_1\|,\ldots,\|\delta_C\|\}$. Indeed, if
$w(S,\HF)=\gamma=f(\delta)$, then $w(S',H'F')=\delta\delta_i$ for some
$i$, and when\vspace*{0.5pt} $\delta_i=w_i(S',H'F')$ of length $\leq K$,
one has $ww_i^{-1}(S',H'F')=\delta $. This shows (2)
\[
\E_{\m^{\ast n}}\|\gamma\| \leq\E_{\m'^{\ast n}}\|\delta\| \leq
\E_{\m^{\ast n}}\|\gamma\|+K.
\]

Now fix $n$, and define for any $\gamma$ in $\G$ the measure with
support in \mbox{$f^{-1}(\gamma) \subset\D$}
\[
\nu_\gamma(\delta)=\cases{ %
\displaystyle \frac{\m'^{\ast n}(\delta)}{\sum_{\delta' \in
f^{-1}(\gamma)}\m'^{\ast n}(\delta')}, &\quad if $f(
\delta)=\gamma$,
\vspace*{2pt}\cr
0, &\quad if $f(\delta)\neq\gamma$.}
\]
Then the measure $\m'^{\ast n}$ decomposes as
\[
\m'^{\ast n}=\sum_{\gamma\in\G}
\m^{\ast n}(\gamma) \nu_\gamma
\]
and by Lemma A.4 in \cite{BKN} on conditionnal entropy,
\[
H\bigl(\m'^{\ast n}\bigr) \leq\sum
_{\gamma\in\G} \m^{\ast n}(\gamma) H(\nu_\gamma) + H
\bigl(\m^{\ast n}\bigr).
\]
The support of $\nu_\gamma$ has size $\leq C$ so $H(\nu_\gamma)
\leq\log C$, which shows (3)
\[
H\bigl(\m^{\ast n}\bigr) \leq H\bigl(\m'^{\ast n}
\bigr) \leq\log C + H\bigl(\m^{\ast n}\bigr).
\]
\upqed
\end{pf}
%
\begin{proposition}[(High asymptotic)]\label{HA}
If there exists $l$ such that $d'_l=\infty$ and
\[
S_\infty> \bigl\langle S''_l,H''_lF''_l
\bigr\rangle\simeq S_l \ast(H_l\times F_l)
\]
is nonamenable, then $\D(S',H'F')< \operatorname{Aut}(ET_{\bar e})$ is nonamenable.

In particular, for the random walk $Y_n$ of law $\m$ equidistributed
on the finite generating set $S'H'F'S'$, there exists $c>0$ such that
for $n$ large enough:
\begin{longlist}[(3)]
\item[(1)] $P(Y_n=_\D1) \leq e^{-cn}$;
\item[(2)] $L_{\D,\m'}(n) \geq cn$;
\item[(3)] $H_{\D,\m'}(n) \geq cn$.
\end{longlist}
\end{proposition}

The proof will use the following:
%
\begin{fact}\label{inf}
Given $\gamma_1 \in\D_1$, there exists $\gamma$ in $\D$ and some
$\gamma_2,\ldots,\gamma_{e_0}$ in $\D_1$ such that $\gamma=(\gamma_1,\gamma_2,\ldots,\gamma_{e_0})id_{S'}$.
\end{fact}

Note that this fact implies that $\D$ is infinite, for it contains
$\gamma
s$ for all $s$ in $S'$ and then $\# \D\geq\#S' \#\D_1 \geq\#S'
\#S'_1 \#\D_2 \geq\cdots\,$. This is in particular an elementary proof
that directed saturated groups are infinite.
\begin{pf*}{Proof of Fact \ref{inf}}
Let $\gamma_1=x_1\cddots x_r$ be a representative word in $S'_1 \sqcup
H'_1F'$. By saturation of $H$, there exists for each $x_i$, an element
$h_i'$ in $H'$ such that $h_i'=(\ast,\ldots,x_i,\ldots,\ast)$ (where
$\ast$ marks some unknown value) with $x_i$ in position $1$ if $x_i
\in H'_1F'$ and $x_i$ in some position between $2$ and $d_0$ if $x_i
\in S'_1$. Now by saturation of $\G$, $S=S_{d_0}$ so $S'$ acts
transitively on $\{1,\ldots,d_0\}$, and there exists $s_i \in S'$ such
that $y_i=s_ih_is_i^{-1}=(x_i,\ast,\ldots,\ast)$. Then $\gamma
=y_1\cddots y_r=(x_1\cddots x_r,\ast,\ldots,\ast)=(\gamma_1,\gamma_2,\ldots,\gamma_{e_0})$. (Note that in fact $\gamma_{d_0+1}=\cdots
=\gamma_{e_0}=1$.)
\end{pf*}
\begin{pf*}{Proof of Proposition \ref{HA}}
The fact shows that the composition
\[
St_1(\D) \hookrightarrow\D_1 \times\cdots\times\D_1
\xrightarrow{\mathrm{pr}_1} \D_1
\]
is surjective, so that if
$\D_1$ is nonamenable, so is $St_1(\D)$ which is a subgroup of $\D$,
and thus $\D$ is nonamenable. Iterating the process shows that if
$\D_l$ is nonamenable, so is $\D$. Consequence (1) follows by Kesten's
theorem \cite{Kes}, (2) and (3) by the Kaimanovich--Vershik theorem
\cite{KV}.
\end{pf*}

\subsection{High order oscillations} The following theorems are
similar to Theorem~7.1 in \cite{Bri3} on oscillation of growth
functions; see also Chapter 2 in \cite{Bri1} and~\cite{KP}. The
entropy function of the groups $\D$ involved is not precisely
evaluated, but only some (rare) values of the function. The idea is to
use alternatingly localization and asymptotic evaluation to obtain a
group with low entropy at some scales and high entropy at other scales.
%
\begin{theorem}\label{thmh}
Let $\G(S,\HF)<\operatorname{Aut}(ET_{\bar d})$ be a saturated directed group with
measure $\m$ equidistributed on $S\HF S$. Let $h_1(n),h_2(n)$ be
functions such that $\frac{h_1(n)}{H_{\G,\m}(n)}\rightarrow\infty$
and $\frac{h_2(n)}{n}\rightarrow0$. Then there exists a group $\D
(S',H'F')<\operatorname{Aut}(ET_{\bar e})$ such that the entropy sequence for the
measure $\m'$ equidistributed on $S'H'F'S'$ satisfies:
\begin{longlist}[(3)]
\item[(1)] $H_{\G,\m}(n) \leq H_{\D,\m'}(n) \leq Cn$ for all $n$
and a constant $C$;
\item[(2)] $H_{\D,\m'}(n_i) \leq h_1(n_i)$ for an infinite sequence $(n_i)$;
\item[(3)]$H_{\D,\m'}(m_i) \geq h_2(m_i)$ for an infinite sequence $(m_i)$.
\end{longlist}
\end{theorem}
%
\begin{corollary}\label{corh}
For any $\frac{1}{2} \leq\a\leq1$, there exists a finitely
generated group $\D$ and a finitely supported symmetric measure $\m'$
such that
\[
\underline{h}\bigl(\D,\m'\bigr)=\a\quad\mbox{and}\quad \overline{h}
\bigl(\D,\m'\bigr)=1.
\]
\end{corollary}

Corollaries \ref{exactentropy}, \ref{oscil} and \ref{corh} imply
Theorem \ref{MT}.
\begin{pf*}{Proof of Corollary \ref{corh}}
If $\a=1$, take $\D$ any nonamenable group. If $\a<1$, apply
Theorem \ref{thmh} with $h(\G,\m)=\a$ (exists by Corollary \ref
{exactentropy}), $h_1(n)=H_{\G,\m}(n) \log n$ and $h_2(n)=\frac
{n}{\log n}$.
\end{pf*}
\begin{pf*}{Proof of Theorem \ref{thmh}}
The first condition is trivially satisfied since $\G(S,\HF)$ is a
quotient of $\D(S',H'F')$.

The strategy is to construct the group $\D(S',H'F')<\operatorname{Aut}(ET_{\bar e})$
with a sequence $(l_i)_i$ rapidly increasing such that $d_l'=0$ when $l
\notin\{l_i\}_i$, and $(d'_{l_i})_i$ is rapidly increasing such that
the group $\langle S''_{l_i},H''_{l_i}F''_{l_i} \rangle<S_{d'_{l_i}}$
is a big finite quotient of the free product $S''_{l_i}\ast H''_{l_i}F''_{l_i}$.

Roughly, as $(l_i)$ is rapidly increasing, there are scales at which an
observer has the impression that $\D(S',H'F')$ resembles the group $\G
(S,\HF)$ and has low asymptotic. As $d'_{l_i}$ is big, there are scales
at which an observer has the impression that $\D(S',H'F')$ contains a
free group and has high asymptotic.

More precisely, assume by induction that parameters $l_j,d'_{l_j}$ and
integers $m_{j-1} \leq n_j \leq m_j$ are constructed for $j <i$
together with an integer $k_{i-1}=1+\log_2 r_{i-1} \geq l_{i-1}$ such
that $\{H(m)|m \leq m_{i-1}\}$ depends only on $B(r_{i-1})$. By
localization (Proposition \ref{local}), this ball, and hence the values of the
$m_{i-1}$ first values of the entropy function, depend uniquely on
$l,d'_{l}$ and the groups $\langle S''_{l},H''_{l}F''_{l} \rangle
<S_{d'_{l}}$ for $l\leq k_{i-1}$.

Now construct $l_i,d'_{l_i},n_i,m_i,k_i$ by describing the sequence of
finite groups $\langle S''_{l},H''_{l}F''_{l} \rangle<S_{d'_{l}}$ for
$k_{i-1}<l \leq k_i$.

The group $\underline{\D}_i(\G',H'F')$ where $d'_l=0$ for all $l\geq
k_i$ has low asymptotic by Proposition \ref{LA}, so
%
\begin{equation}
\label{LAE} H_{\underline{\D}_i,\m'}(n) \leq H_{\G,\m}(n)+\log C
\end{equation}
for some $C$, and in particular, there exists $n_i$ such that
\[
H_{\underline{\D}_i,\m'}(n_i) \leq h_1(n_i).
\]
By localization (Proposition \ref{local}) this value of entropy depends only on a
ball of radius $R_i$ in the Cayley graph of $\underline{\D}_i(\G',H'F')$ which depends only on the $L_i=1+\log_2 R_i$ first levels.

Now fix $l_i=L_i+1$ and let $\overline{\D}_i(S',H'F')$ be the group
with $d'_l$ and $\langle S''_{l},\allowbreak H''_{l}F''_{l} \rangle<S_{d'_{l}}$ as
above for $l \leq l_i-1$ and fix (momentarily) $d'_{l_i}=\infty$, with
\[
\bigl\langle S''_{l_i},H''_{l_i}F''_{l_i}
\bigr\rangle\simeq S''_{l_i} \ast
H''_{l_i}F''_{l_i}
\simeq S_l \ast H_lF <S_\infty
\]
and $d'_l=0$ for $l >l_i$. The group $\overline{\D}_i(S',H'F')$ is
nonamenable by Proposition \ref{HA} of high asymptotic, so
%
\begin{equation}
\label{HAE} H_{\overline{\D}_i,\m'}(m) \geq cm
\end{equation}
for some $c>0$, and in particular, there exists $m_i$ such that
\[
H_{\overline{\D}_i,\m'}(m_i) \geq h_2(m_i).
\]
Now by localization (Proposition \ref{local}), the $m_i$ first values of the entropy
function depend only on a ball of radius $r_i$ in the Cayley graph of
$\overline{\D}_i(S',H'F')$, which depends only on the $k_i=1+\log_2
r_i$ first levels, and the balls of radius $2r_i+1$ in $\langle
S''_{l},H''_{l}F''_{l} \rangle<S_{d'_{l}}$ for $l \leq k_i$.

In particular, the values $\{H(m)|m\leq m_i\}$ are the same if $\langle
S''_{l_i},H''_{l_i}F''_{l_i} \rangle$ is any group coinciding with the
free product $S''_{l_i} \ast H''_{l_i}F''_{l_i}\simeq S_l \ast H_lF$ in
a ball of radius $2r_i+1$. As a free product of finite groups is
residually finite, there exists such a group which is finite of size $d'_{l_i}$.

The parameters $l_i,d'_{l_i},k_i\geq l_i$, the finite groups $\langle
S''_{l},H''_{l}F''_{l} \rangle< S_{d'_l}$ for $l \leq k_i$ and the
integers $n_i,m_i$ are now constructed for all $i$ by induction.

The sequence $(d'_{l_i})_i$ of positive integers and the finite groups
$\langle S''_{l_i},H''_{l_i}F''_{l_i} \rangle< S_{d'_{l_i}}$ allow us
to define a group $\D$ with entropy satisfying $H_{\D,\m'}(n_i) \leq
h_1(n_i)$ and $H_{\D,\m'}(m_i) \geq h_2(m_i)$ for all $i$, because
the balls of radius $r_i$ in $\D$ and in $\overline{\D}_i$ coincide.
\end{pf*}

In the point of view of information theory of Remark \ref{IT}, the
minimal tree description remains valid. However, the number of digits
needed to describe $\s_u$ is not anymore bounded independently of the
level $l$, for now $d_{l_i} \rightarrow\infty$, which explains the
higher values taken by the entropy.
%
\begin{theorem}\label{thmp}
Let $\G(S,\HF)<\operatorname{Aut}(ET_{\bar d})$ be a saturated directed group with
random walk $Y_n$ of law $\m$ equidistributed on $S\HF S$. Let
$p_1(n),p_2(n)$ be functions such that $\frac{p_1(n)}{P(Y_n=_\G
1)}\rightarrow0$ and that for any $c>0$ and $n$ large, $p_2(n) \geq
e^{-cn}$. Then there exists a group $\D(S',H'F')<\operatorname{Aut}(ET_{\bar e})$
such that the return probability for the random walk $Y'_n$ with law
$\m'$ equidistributed on $S'H'F'S'$ satisfies:
\begin{longlist}[(3)]
\item[(1)] $P(Y_n=_\G1) \geq P(Y'_n=_\D1) \geq e^{-Cn}$ for all $n$
and a constant $C$;
\item[(2)] $P(Y'_{n_i}=_\D1) \geq p_1(n_i)$ for an infinite sequence $(n_i)$;
\item[(3)] $P(Y'_{n_i}=_\D1) \leq p_2(m_i)$ for an infinite sequence $(m_i)$.
\end{longlist}
\end{theorem}
%
\begin{theorem}\label{thml}
Let $\G(S,\HF)<\operatorname{Aut}(ET_{\bar d})$ be a saturated directed group with
measure $\m$ equidistributed on $S\HF S$. Let $L_1(n),L_2(n)$ be
functions such that $\frac{L_1(n)}{L_{\G,\m}(n)}\rightarrow\infty$
and $\frac{L_2(n)}{n}\rightarrow0$. Then there exists a group $\D
(S',H'F')<\operatorname{Aut}(ET_{\bar e})$ such that the drift for the measure $\m'$
equidistributed on $S'H'F'S'$ satisfies:
\begin{longlist}[(3)]
\item[(1)] $L_{\G,\m}(n) \leq L_{\D,\m'}(n) \leq Cn$ for all $n$
and a constant $C$;
\item[(2)] $L_{\D,\m'}(n_i) \leq L_1(n_i)$ for an infinite sequence $(n_i)$;
\item[(3)] $L_{\D,\m'}(m_i) \geq L_2(m_i)$ for an infinite sequence $(m_i)$.
\end{longlist}
\end{theorem}
\begin{pf}
The proof of Theorem \ref{thmh} applies to Theorems \ref{thmp} and
\ref{thml}, with (a priori) different choices of parameters\vspace*{1pt}
$l_i,d'_{l_i}$ and integers $n_i,m_i$, obtained by replacing inequality
(\ref{LAE}) by (Proposition \ref{LA} of low asymptotic)
%
\begin{eqnarray}
P(Y_n=_\D1) &\geq&\frac{1}{C} P(Y_n=_\G1)
\geq p_1(n),
\\
L_{\D,\m'}(n) &\leq& L_{\G,\m}(n)+K \leq L_1(n)
\end{eqnarray}
for $n$ large enough, which allows us to find $n_i$, and replacing
inequality (\ref{HAE}) by (Proposition \ref{HA} of high asymptotic)
%
\begin{eqnarray}
P(Y_m=_\D1) &\leq& e^{-cm} \leq
p_2(m),
\\
L_{\D,\m'}(m) &\geq& cm \geq L_2(m)
\end{eqnarray}
for $m$ large enough, which allows us to find $m_i$.
\end{pf}
%
\begin{corollary}[(Theorem \ref{MTT})]\label{corp}
There exists a finitely generated group $\D$ and a finite symmetric
measure $\m'$ such that the return probability of the random walk
$Y'_n$ of law $\m'$ satisfies
\[
\underline{p}(\D)=\tfrac{1}{3},\qquad \overline{p}(\D)=1 \quad\mbox{and}
\quad\underline{\delta}\bigl(\D,\m'\bigr)=\tfrac{1}{2},\qquad \overline{
\delta}\bigl(\D,\m'\bigr)=1.
\]
\end{corollary}
\begin{pf}
Take $\G(S,\HF)=F \wr_{\partial T_2} D_\infty<\operatorname{Aut}(ET_2)$ to be the
directed saturated group of Proposition \ref{specific}, for which
$e^{-n^{{1/3}-\varepsilon}} \geq P(Y_n=_\G1) \geq e^{-n^{
{1/3}}}$ and $L_{\G,\m}(n) \approx n^{{1/2}}$. Take
$p_1(n)=e^{-n^{{1/3}}\log n}$, $p_2(n)=e^{-{n}/{\log n}}$,
$L_1(n)=\break n^{{1/2}}\log n$ and $L_2(n)=\frac{n}{\log n}$. Apply
Theorems \ref{thmp} and \ref{thml}.
\end{pf}

\section{Generalization}\label{gls}

The definition of a saturated directed group acting on an extended tree
$ET_{\bar d}$ for a bounded sequence $\bbar d$ can be slightly
generalized, with adaptation of Theorems \ref{MTE} and \ref{RPE}.

By $S_d$ denote a finite group acting faithfully and transitively on $\{
1,\ldots,d\}$ (not necessarily the full group of permutation). Replace
definition (\ref{defh}) in Section \ref{directedgroups} by
\[
h=(h_1,\s_2,\ldots,\s_{d_0})\s_h
\]
with $\s_h$ in $\mathrm{Fix}_{S_{d_0}}(1)$, a permutation of $\{1,\ldots,d_0\}$
that fixes $1$. By recursion, $\s_{h_l}(1)=1$ for all $l$, and hence
$h$ fixes the infinite ray $1^\infty$, and thus commutes with $\varphi_f$. The group $H_{\bar d}$ directed by the ray $1^\infty$ is the
uncountable locally finite group $H_{\bar d}=AT_0 \times AT_1 \times
\cdots\,$, where
\[
AT_l=S_{d_{l+1}} \wr S_{d_l-1}=(\underbrace{S_{d_{l+1}}
\times\cdots \times S_{d_{l+1}}}_{d_l-1\ \mathrm{factors}}) \rtimes S_{d_l-1}.
\]
Given another sequence $\bbar c=(c_l)_l$ of integers such that $1 \leq
c_l \leq d_l-1$, define the subgroup $PT_l$ by
\[
PT_l=\underbrace{S_{d_{l+1}} \times\cdots\times
S_{d_{l+1}}}_{c_l\
\mathrm{factors}} \times\{1\} \times\cdots\times\{1\} <
S_{d_{l+1}} \wr S_{d_l-1}=AT_l
\]
with $c_l$ factors $S_{d_{l+1}}$ when $c_l < d_l-1$, and $PT_l=AT_l$ if
$c_l=d_l-1$. The hypothesis of saturation of a group $G(S,H)$ can be
relaxed as relative saturation with respect to $\bbar c$ by requiring
that $S=S_{d_0}$ acts transitively on $\{1,\ldots,d_0\}$, and the
subgroup $H<H_{\bar d}$ is included in
\[
H <PT_0 \times PT_1 \times\cdots
\]
with the equidistribution measure on $H$ projecting to equidistribution
measure on each factor $PT_l$.

Given the sequences $\bbar d=(d_l)_l$ and $\bbar c=(c_l)_l$ of integers
with $1 \leq c_l \leq d_l-1$, define a new sequence $\bbar
{p}'=(p'_l)_l$ by $p'_l=\frac{c_l}{(c_l+1)d_l}$, and set
\begin{eqnarray*}
\b_{\bar d,\bar c}(n)&=&\frac{\log(d_0\cddots
d_{k(n)})}{\log n}\qquad \mbox{where }
k(n)=k_{\bar d, \bar c}(n) =\min\bigl\{k|p'_0\cddots
p'_k n \leq1\bigr\},
\\
\b'_{\bar d,\bar c}(n)&=&\frac{\log(d_0\cddots d_{l(n)})}{\log n}
\qquad\mbox{where }
l(n)=l_{\bar d,\bar c}(n)=\max\biggl\{l\Big|\frac
{d_0}{p_0'}\cddots
\frac{d_l}{p_l'}\leq n\biggr\}.
\end{eqnarray*}
With this notation, Theorems \ref{MTE} and \ref{RPE} generalize to:
%
\begin{theorem}\label{gl}
Given bounded sequences $\bbar d$ and $\bbar c$, a relatively saturated
directed group $\G(S,\HF)$ and the measure $\m$ of equidistribution on
$S\HF S$, one has for arbitrary $\varepsilon>0$ and $n$ large:
\begin{longlist}[(3)]
\item[(1)] $|\frac{\log H_{\G,\m}(n)}{\log n}-\b_{\bar d,\bar
c}(n)|\leq\varepsilon$;
\item[(2)] $\b_{\bar d,\bar c}(n)-\varepsilon\leq\frac{\log L_{\G,\m
}(n)}{\log n} \leq\frac{1+\b_{\bar d,\bar c}(n)}{2}+\varepsilon$;
\item[(3)] $\b'_{\bar d,\bar c}(n)-\varepsilon\leq\frac{\log\log
P(Y_n=_\G1)}{\log n} \leq\b_{\bar d,\bar c}(n)+\varepsilon$.
\end{longlist}
\end{theorem}

For constant sequences $d_l=d$ and $c_l=c\leq d-1$, the sequences $\b_{\bar d,\bar c}(n)$ and $\b'_{\bar d,\bar c}(n)$ have limits, respectively,
\begin{eqnarray*}
\b_{d,c}&=&\frac{\log d}{\log({d(c+1)}/{c})}\\
&=&\frac{1}{1+
{\log(({c+1})/{c})}/{\log d}}\qquad
\mbox{so }
\b_{d,1}=\frac
{1}{1+{\log2}/{\log d}},
\\
\b'_{d,c}&=&\frac{\log d}{\log({d^2(c+1)}/{c})}\\
&=&\frac{1}{2+
{\log(({c+1})/{c})}/{\log d}}
\qquad\mbox{so }\b'_{d,1}=\frac
{1}{2+{\log2}/{\log d}}.
\end{eqnarray*}
%
\begin{examples}\label{ex}
(1) When the valency sequence is constant of value $d$, for
$H=AT(d)=S_d \wr S_{d-1}$ diagonally embedded into $AT_0 \times AT_1
\times\cdots\,$, obtain the mother group $G(S_d,AT(d))$ of polynomial
automata of degree $0$; see \cite{AAV,BKN}. With an extension
$F$ at the tree boundary, one has for $\G(d)=\G
(S_d,AT(d)F)<\operatorname{Aut}(ET_d)$ the estimates
\[
h\bigl(\G(d),\m\bigr)=\b_d \quad\mbox{and}\quad \b'_d
\leq\underline{p}\bigl(\G (d),\m\bigr) \leq\overline{p}\bigl(\G(d),\m\bigr) \leq
\b_d.
\]

(2) For the spinal groups $G_\o(q)$ of the article \cite{BS}
acting on a tree of constant valency $q=d$. With an extension $F$ at
the tree boundary, one has for $\G_\o(q)=F \wr_{\partial T} G_\o
(q)$ the estimates
\[
h\bigl(\G_\o(q),\m\bigr)=\b_{q,1} \quad\mbox{and}\quad
\b'_{q,1} \leq\underline {p}\bigl(\G_\o(q),\m
\bigr) \leq\overline{p}\bigl(\G_\o(q),\m\bigr) \leq
\b_{q,1}.
\]
\end{examples}
\begin{pf*}{Proof of Theorem \ref{gl}}
Notation of Proposition \ref{rp} and Lemma \ref{randomRP} becomes
\[
k_j=\bigl(k'_j,b_{j,2},
\ldots,b_{j,c_0+1},1,\ldots,1\bigr)\s_{h_j}
\]
and one should consider $k_j^{\s_j}$ for $\s_j=s_1 \s_{h_1} s_2
\cddots s_{j-1}\s_{h_{j-1}}s_j$. The sequence $(\s_j)_j$ consists of
independent terms equidistributed in $S_{d_0}$. As its action is
transitive on $\{1,\ldots,d_0\}$, for any fixed $t$, the sequence $(\s_j(t))_j$ is a sequence of independent equidistributed elements of $\{
1,\ldots,d_0\}$.

This ensures that $Y_n^t$ is a product of $n$ terms that are:
\begin{longlist}[(3)]
\item[(1)] either $b_{j,\s_j(t)}$ equidistributed in $S_{d_1}$ at
times $j$ when $\s_j(t) \in\{2,\ldots,\break c_0+1\}$,
\item[(2)]or $k'_j$ equidistributed in $H_1F$ at times $j$ when $\s_j(t)=1$,
\item[(3)] or trivial factors $1$ at times $j$ when $\s_j(t) \in\{
c_0+2,\ldots,d_0\}$.
\end{longlist}

The number of nontrivial terms is $n_t \sim\frac{c_0+1}{d_0}n$
almost surely. Now merging packs of terms in the same finite group
$S_{d_1}$ (with probability $\frac{c_0}{c_0+1}$ among nontrivial
terms) or $H_1F$ (with probability $\frac{1}{c_0+1}$ among nontrivial
terms), the length of $Y_n^t$ is almost surely
\[
m_t \sim\frac{c_0}{(c_0+1)^2}n_t \sim\frac{c_0}{(c_0+1)d_0}n.
\]
This shows that Lemma \ref{randomRP} is true with $p_0'=\frac
{c_0}{(c_0+1)d_0}$ instead of $p_0=\frac{d_0-1}{d_0^2}$.

Then Lemma \ref{expectact} generalizes as
\[
\biggl\llvert \frac{\log\E a(Y_n)}{\log n}-\b_{\bar d, \bar c}(n)\biggr\rrvert \leq
\varepsilon
\]
for $\varepsilon>0$ and $n$ large. Theorems \ref{MTE} and \ref{RPE}
and Corollary \ref{cordrift} follow straightforward. Note that the
condition on $\th$ in Fact \ref{vercon} becomes $p_i'-\th> \frac
{p_i'}{d_i}$.
\end{pf*}

\section{Comments and questions}\label{cq}

\subsection{Analogies between growth and entropy for directed groups}
Analogies between growth and entropy for directed groups are two fold.

First, there is an analogy between the computation of growth exponents
[as $\lim\inf \underline{\a}(\G)$, $\lim\sup \overline{\a}(\G)$ or
limit $\a(\G)$ of $\frac{\log\log b_\G(r)}{\log r}$] in \cite
{Bri3} and the computation of entropy exponents in Theorem \ref{MTE}.
For entropy, the computation is based on the contraction by a factor
$p_0$ of the word length under rewriting process of random alternate
words in $\G(S,\HF)$. For growth in the extended Aleshin--Grigorchuk
group $\G_{(012)^\infty}=F \wr G_{(012)^\infty}$, the computation is
based on the contraction by a factor $\frac{\eta}{2}$ in the wreath
product for reduced representative words; see \cite{BE}, and Lemma 5.4
in~\cite{Bri3}.

The contraction factor for entropy should only hold for random
alternate words, whereas the contraction factor for growth has to hold
for any alternate (pre-reduced) word, which heuristically explains why
$\frac{1}{2} =h(\G_{(012)^\infty},\m) < \a(\G_{(012)^\infty})
\approx0.76$. This inequality is a well-known property of Shannon
entropy that $H(\m) \leq\log\#\operatorname{supp}(\m)$ with equality for an
equidistributed measure \cite{Sha}.

There is a second analogy at the level of parameter space. For a fixed
bound $D$ on the valency $\bbar d$, the space of saturated directed
groups is (partially) parametrized by the Cantor set $\{2,\ldots,D\}^\N
$. The entropy exponent $\b_{\bar d}(n)$ is computed in Theorem \ref
{MTE} in terms of the sequence $\bbar d$ and the contraction factors
$(p_i)_i$ as $n^{\b_{\bar d}(n)} = d_0\cddots d_{k(n)}$ where $p_0\cddots
p_{k(n)} \approx\frac{1}{n}$. By Fact \ref{classfunction}, any
function $g(n)$ such that $dg(n) \leq g(\frac{d^2}{d-1}n)$ and
$g(\frac{D^2}{D-1}n)\leq Dg(n)$
is the entropy\vspace*{1pt} of some finitely generated group (with the approximation
of Theorem \ref{MTE}).

The space of extended Aleshin--Grigorchuk groups $\G_\o$ is also
parametrized by a Cantor set $\{0,1,2\}^\N$. Though the growth
function for a given sequence $\o$ is not known, Bartholdi and
Erschler have shown recently in \cite{BE3} that any function
$e^{g(n)}$ with $g(2n) \leq2g(n) \leq g(\eta_+ n)$ for $\eta_+
\approx2.46$ explicit is the growth function of some group (also
compare Corollary 4.2 in \cite{BE3} with definition of exponent
sequence at Section \ref{defbeta}).

\subsection{Comparison between growth, entropy, return probability and drift}

\mbox{\hspace*{-4pt}}Among finitely generated groups with symmetric finitely supported
measure, it is a natural question to classify the pairs $(\underline
{\a}(\G),\overline{\a}(\G))$, $(\underline{h}(\G,\m),\overline
{h}(\G,\m))$, $(\underline{p}(\G),\overline{p}(\G))$ and
$(\underline{\delta}(\G,\m),\overline{\delta}(\G,\m))$ in the
triangle $0 \leq\a\leq\b\leq1$. Comparing Theorem \ref{MTT} with
Theorem \ref{MT}, and the main result in \cite{Bri3} raises the
following two questions.
%
\begin{question}
Given a point $(\a,\b)$ in the triangle $\frac{1}{3} \leq\a\leq\b
\leq1$, does there exist a finitely generated group $\G$ the return
probability exponents of which satisfy
\[
\bigl(\underline{p}(\G),\overline{p}(\G)\bigr)=(\a,\b)\mbox{?}
\]
\end{question}
%
\begin{question}
Given a point $(\a,\b)$ in the triangle $\frac{1}{2} \leq\a\leq\b
\leq1$, does there exist a finitely generated group $\G$ together
with a (symmetric finitely supported) measure $\m$, the drift
exponents of which satisfy
\[
\bigl(\underline{\delta}(\G,\m),\overline{\delta}(\G,\m)\bigr)=(\a,\b)
\mbox{?}
\]
\end{question}

One approach would be to improve Theorem \ref{RPE} by understanding
how a particular choice of $H$ affects the return probability. Another
approach is by the technics developped in \cite{KP} that could lead to
strengthening of Theorems \ref{thmh} and \ref{thmp}.

Another natural question is to know if there are such pairs outside the
above mentioned triangles besides the pair $(0,0)$, obtained by
virtually nilpotent groups for growth, entropy and return probability,
by finite groups for drift. By \cite{LP}, the number $\frac{1}{2}$ is
a lower bound on the drift exponent of infinite groups. It raises the:
%
\begin{question}\label{qlb}
Does there exist a group $\G$ and a measure $\m$ such that
\[
0< \underline{h}(\G,\m) < \tfrac{1}{2}
\]
or
\[
0< \underline{p}(\G) < \tfrac{1}{3}
\]
or
\[
0<\underline{\a}(\G)<\a(\G_{(012)^\infty}) \approx0.76\mbox{?}
\]
\end{question}

By \cite{CGP}, groups of exponential growth have return probability
exponents $\geq\frac{1}{3}$. By Theorem \ref{RPE}, this is also a
lower bound for many groups with intermediate growth. A conjecture of
Grigorchuk asserts that a finitely generated group $\G$ is virtually
nilpotent when its growth function satisfies $b_\G(r) \leq e^{r^{
{1/2}-\varepsilon}}$. If this were the case, then $\overline{p}(\G
)<\frac{1}{5}$ would imply virtual nilpotency by \cite{CGP}. The
bound $\frac{1}{2}$ for entropy corresponds to some simple random walk
on the lamplighter group or on an extended directed Aleshin--Grigorchuk
group; see Remark \ref{entgro}.

Finally, one may wonder how these asymptotic quantities relate with
each other. For instance entropy is bounded by logarithm of growth so
$h(\G,\m) \leq\a(\G)$. Also Corollary 7.4 in \cite{CGP} implies that
\[
\frac{\underline{\a}(\G)}{2+\underline{\a}(\G)}
\leq\underline {p}(\G) \quad\mbox{and}\quad \overline{p}(\G) \leq
\frac{\overline{\a
}(\G)}{2-\overline{\a}(\G)}.
\]
Naively, one expects groups with low growth to have low return
probability exponents and vice-versa. However, taking $\G$ the
lamplighter on $\Z$ and $\G'=F \wr G_\o(q)$ an extension of a spinal
group $G_\o(q)$ of the article \cite{BS}, one has (note that
Theorem~6.1 in \cite{BS} still applies with boundary extension $F$, and
see Example \ref{ex}(2))
\[
\overline{\a}\bigl(\G'\bigr)<\a(\G)=1 \qquad\mbox{but }
\frac{1}{3} = p(\G) < \frac{1}{2+{\log2}/{\log q}}
\leq\underline{p}\bigl(
\G'\bigr).
\]
This raises the following question:
%
\begin{question}
For $\G$ finitely generated group with finitely supported symmetric
measure $\m$, what are the possible values of the $8$-tuple
\[
\bigl(\underline{\a}(\G),\overline{\a}(\G),\underline{p}(\G ),\overline{p}(\G),
\underline{h}(\G,\m),\overline{h}(\G,\m ),\underline{\delta}(\G,\m),\overline{
\delta}(\G,\m)\bigr)\mbox{?}
\]
\end{question}

\section*{Acknowledgments}

I am most grateful to Gidi Amir and Balint Virag who helped me correct
and improve a preliminary nonaccurate version of this article. I also
wish to thank Laurent Saloff-Coste, Alain Valette and Andrzej Zuk, who
supported me during this work. Finally, I thank the anonymous referee
for careful reading and many comments.



\printaddresses


\begin{thebibliography}{45}

\bibitem{Ale}
\begin{barticle}[mr]
\bauthor{\bsnm{Ale{\v{s}}in},~\bfnm{S.~V.}\binits{S.~V.}}
(\byear{1972}).
\btitle{Finite automata and the {B}urnside problem for periodic groups}.
\bjournal{Mat. Zametki}
\bvolume{11}
\bpages{319--328}.
\bid{issn={0025-567X}, mr={0301107}}
\bptok{imsref}%
\end{barticle}
\endbibitem

\bibitem{AAV}
\begin{bmisc}[auto:STB|2012/08/21|06:38:10]
\bauthor{\bsnm{Amir},~\bfnm{G.}\binits{G.}},
  \bauthor{\bsnm{Angel},~\bfnm{O.}\binits{O.}} \AND
  \bauthor{\bsnm{Virag},~\bfnm{B.}\binits{B.}}
(\byear{2009}).
\bhowpublished{Amenability of linear-activity automaton groups. Available at
  arXiv:\arxivurl{0905.2007v1}.}
\bptok{imsref}%
\end{bmisc}
\endbibitem

\bibitem{AV}
\begin{bmisc}[auto:STB|2012/08/21|06:38:10]
\bauthor{\bsnm{Amir},~\bfnm{G.}\binits{G.}} \AND
  \bauthor{\bsnm{Virag},~\bfnm{B.}\binits{B.}}
(\byear{2011}).
\bhowpublished{Positive speed for high-degree automaton groups. Available at
  arXiv:\arxivurl{1102.4979v1}.}
\bptok{imsref}%
\end{bmisc}
\endbibitem


\bibitem{BE3}
\begin{bmisc}[auto:STB|2012/08/21|06:38:10]
\bauthor{\bsnm{Bartholdi},~\bfnm{L.}\binits{L.}} \AND
  \bauthor{\bsnm{Erschler},~\bfnm{A.}\binits{A.}}
(\byear{2011}).
\bhowpublished{Groups of given intermediate word growth. Available at
  arXiv:\arxivurl{1110.3650v1}.}
\bptok{imsref}%
\end{bmisc}
\endbibitem

\bibitem{BE}
\begin{barticle}[mr]
\bauthor{\bsnm{Bartholdi},~\bfnm{Laurent}\binits{L.}} \AND
  \bauthor{\bsnm{Erschler},~\bfnm{Anna}\binits{A.}}
(\byear{2012}).
\btitle{Growth of permutational extensions}.
\bjournal{Invent. Math.}
\bvolume{189}
\bpages{431--455}.
\bid{doi={10.1007/s00222-011-0368-x}, issn={0020-9910}, mr={2947548}}
\bptnote{check year}%
\bptok{imsref}%
\end{barticle}
\endbibitem

\bibitem{BKN}
\begin{barticle}[mr]
\bauthor{\bsnm{Bartholdi},~\bfnm{Laurent}\binits{L.}},
  \bauthor{\bsnm{Kaimanovich},~\bfnm{Vadim~A.}\binits{V.~A.}} \AND
  \bauthor{\bsnm{Nekrashevych},~\bfnm{Volodymyr~V.}\binits{V.~V.}}
(\byear{2010}).
\btitle{On amenability of automata groups}.
\bjournal{Duke Math. J.}
\bvolume{154}
\bpages{575--598}.
\bid{doi={10.1215/00127094-2010-046}, issn={0012-7094}, mr={2730578}}
\bptok{imsref}%
\end{barticle}
\endbibitem

\bibitem{BV}
\begin{barticle}[mr]
\bauthor{\bsnm{Bartholdi},~\bfnm{Laurent}\binits{L.}} \AND
  \bauthor{\bsnm{Vir{\'a}g},~\bfnm{B{\'a}lint}\binits{B.}}
(\byear{2005}).
\btitle{Amenability via random walks}.
\bjournal{Duke Math. J.}
\bvolume{130}
\bpages{39--56}.
\bid{doi={10.1215/S0012-7094-05-13012-5}, issn={0012-7094}, mr={2176547}}
\bptok{imsref}%
\end{barticle}
\endbibitem

\bibitem{BS}
\begin{barticle}[mr]
\bauthor{\bsnm{Bartholdi},~\bfnm{Laurent}\binits{L.}} \AND
  \bauthor{\bsnm{{\v{S}}uni{\'k}},~\bfnm{Zoran}\binits{Z.}}
(\byear{2001}).
\btitle{On the word and period growth of some groups of tree automorphisms}.
\bjournal{Comm. Algebra}
\bvolume{29}
\bpages{4923--4964}.
\bid{doi={10.1081/AGB-100106794}, issn={0092-7872}, mr={1856923}}
\bptok{imsref}%
\end{barticle}
\endbibitem

\bibitem{Bri1}
\begin{bmisc}[auto:STB|2012/08/21|06:38:10]
\bauthor{\bsnm{Brieussel},~\bfnm{J.}\binits{J.}}
(\byear{2008}).
\bhowpublished{Croissance et moyennabilit\'e de certains groupes
  d'automorphismes d'un arbre enracin\'e. Th\`ese de doctorat, Univ. D. Diderot
  Paris 7. Available at
  \texttt{\href{http://www.institut.math.jussieu.fr/theses/2008/brieussel/}{http://www.institut.math.jussieu.fr/theses/2008/}
  \href{http://www.institut.math.jussieu.fr/theses/2008/brieussel/}{brieussel/}}.}
\bptok{imsref}%
\end{bmisc}
\endbibitem

\bibitem{Bri2}
\begin{barticle}[mr]
\bauthor{\bsnm{Brieussel},~\bfnm{J{\'e}r{\'e}mie}\binits{J.}}
(\byear{2009}).
\btitle{Amenability and non-uniform growth of some directed automorphism groups
  of a rooted tree}.
\bjournal{Math. Z.}
\bvolume{263}
\bpages{265--293}.
\bid{doi={10.1007/s00209-008-0417-3}, issn={0025-5874}, mr={2534118}}
\bptok{imsref}%
\end{barticle}
\endbibitem

\bibitem{Bri3}
\begin{bmisc}[auto:STB|2012/08/21|06:38:10]
\bauthor{\bsnm{Brieussel},~\bfnm{J.}\binits{J.}}
(\byear{2011}).
\bhowpublished{Growth behaviors in the range $e^{r^\alpha}$. Available at
  arXiv:\arxivurl{1107.1632v1}.}
\bptok{imsref}%
\end{bmisc}
\endbibitem

\bibitem{CCJJV}
\begin{bbook}[mr]
\bauthor{\bsnm{Cherix},~\bfnm{Pierre-Alain}\binits{P.-A.}},
  \bauthor{\bsnm{Cowling},~\bfnm{Michael}\binits{M.}},
  \bauthor{\bsnm{Jolissaint},~\bfnm{Paul}\binits{P.}},
  \bauthor{\bsnm{Julg},~\bfnm{Pierre}\binits{P.}} \AND
  \bauthor{\bsnm{Valette},~\bfnm{Alain}\binits{A.}}
(\byear{2001}).
\btitle{Groups with the {H}aagerup Property: Gromov's A-T-menability}.
\bseries{Progress in Mathematics}
\bvolume{197}.
\bpublisher{Birkh\"auser}, \blocation{Basel}.
\bid{doi={10.1007/978-3-0348-8237-8}, mr={1852148}}
\bptok{imsref}%
\end{bbook}
\endbibitem

\bibitem{CGP}
\begin{barticle}[mr]
\bauthor{\bsnm{Coulhon},~\bfnm{T.}\binits{T.}},
  \bauthor{\bsnm{Grigor'yan},~\bfnm{A.}\binits{A.}} \AND
  \bauthor{\bsnm{Pittet},~\bfnm{C.}\binits{C.}}
(\byear{2001}).
\btitle{A geometric approach to on-diagonal heat kernel lower bounds on
  groups}.
\bjournal{Ann. Inst. Fourier (Grenoble)}
\bvolume{51}
\bpages{1763--1827}.
\bid{issn={0373-0956}, mr={1871289}}
\bptok{imsref}%
\end{barticle}
\endbibitem

\bibitem{Ers3}
\begin{barticle}[mr]
\bauthor{\bsnm{Erschler},~\bfnm{Anna}\binits{A.}}
(\byear{2003}).
\btitle{On drift and entropy growth for random walks on groups}.
\bjournal{Ann. Probab.}
\bvolume{31}
\bpages{1193--1204}.
\bid{doi={10.1214/aop/1055425775}, issn={0091-1798}, mr={1988468}}
\bptok{imsref}%
\end{barticle}
\endbibitem

\bibitem{Ersdrift}
\begin{barticle}[mr]
\bauthor{\bsnm{Erschler},~\bfnm{Anna}\binits{A.}}
(\byear{2004}).
\btitle{On the asymptotics of the rate of departure to infinity}.
\bjournal{J.~Math. Sci. (N.~Y.)}
\bvolume{121}
\bpages{2437--2440}.
\bptok{imsref}%
\end{barticle}
\endbibitem

\bibitem{Ers}
\begin{barticle}[mr]
\bauthor{\bsnm{Erschler},~\bfnm{Anna}\binits{A.}}
(\byear{2005}).
\btitle{On the degrees of growth of finitely generated groups}.
\bjournal{Funktsional. Anal. i Prilozhen.}
\bvolume{39}
\bpages{86--89}.
\bid{doi={10.1007/s10688-005-0055-z}, issn={0374-1990}, mr={2197519}}
\bptok{imsref}%
\end{barticle}
\endbibitem

\bibitem{Ers2}
\begin{barticle}[mr]
\bauthor{\bsnm{Erschler},~\bfnm{Anna}\binits{A.}}
(\byear{2006}).
\btitle{Piecewise automatic groups}.
\bjournal{Duke Math. J.}
\bvolume{134}
\bpages{591--613}.
\bid{doi={10.1215/S0012-7094-06-13435-X}, issn={0012-7094}, mr={2254627}}
\bptok{imsref}%
\end{barticle}
\endbibitem

\bibitem{Ers4}
\begin{barticle}[mr]
\bauthor{\bsnm{Erschler},~\bfnm{Anna}\binits{A.}}
(\byear{2006}).
\btitle{Isoperimetry for wreath products of {M}arkov chains and multiplicity of
  selfintersections of random walks}.
\bjournal{Probab. Theory Related Fields}
\bvolume{136}
\bpages{560--586}.
\bid{doi={10.1007/s00440-005-0495-7}, issn={0178-8051}, mr={2257136}}
\bptok{imsref}%
\end{barticle}
\endbibitem

\bibitem{Ers6}
\begin{binproceedings}[mr]
\bauthor{\bsnm{Erschler},~\bfnm{Anna}\binits{A.}}
(\byear{2010}).
\btitle{Poisson--{F}urstenberg boundaries, large-scale geometry and growth of
  groups}.
In \bbooktitle{Proceedings of the {I}nternational {C}ongress of
  {M}athematicians. {V}olume {II}}
\bpages{681--704}.
\bpublisher{Hindustan Book Agency}, \blocation{New Delhi}.
\bid{mr={2827814}}
\bptok{imsref}%
\end{binproceedings}
\endbibitem

\bibitem{Fol}
\begin{barticle}[mr]
\bauthor{\bsnm{Folner},~\bfnm{Erling}\binits{E.}}
(\byear{1955}).
\btitle{On groups with full {B}anach mean value}.
\bjournal{Math. Scand.}
\bvolume{3}
\bpages{243--254}.
\bid{issn={0025-5521}, mr={0079220}}
\bptok{imsref}%
\end{barticle}
\endbibitem

\bibitem{Gri1}
\begin{barticle}[mr]
\bauthor{\bsnm{Grigorchuk},~\bfnm{R.~I.}\binits{R.~I.}}
(\byear{1985}).
\btitle{Degrees of growth of finitely generated groups and the theory of
  invariant means}.
\bjournal{Math. USSR Izv.}
\bvolume{25}
\bpages{259--300}.
\bptok{imsref}%
\end{barticle}
\endbibitem

\bibitem{Gri2}
\begin{barticle}[mr]
\bauthor{\bsnm{Grigorchuk},~\bfnm{R.~I.}\binits{R.~I.}}
(\byear{1986}).
\btitle{Degrees of growth of {$p$}-groups and torsion-free groups}.
\bjournal{Math. USSR-Sb.}
\bvolume{54}
\bpages{185--205}.
\bptok{imsref}%
\end{barticle}
\endbibitem

\bibitem{GZ}
\begin{barticle}[mr]
\bauthor{\bsnm{Grigorchuk},~\bfnm{Rostislav~I.}\binits{R.~I.}} \AND
  \bauthor{\bsnm{{\. Z}uk},~\bfnm{Andrzej}\binits{A.}}
(\byear{2002}).
\btitle{On a torsion-free weakly branch group defined by a three state
  automaton}.
\bjournal{Internat. J. Algebra Comput.}
\bvolume{12}
\bpages{223--246}.
\bid{doi={10.1142/S0218196702001000}, issn={0218-1967}, mr={1902367}}
\bptok{imsref}%
\end{barticle}
\endbibitem

\bibitem{GS}
\begin{barticle}[mr]
\bauthor{\bsnm{Gupta},~\bfnm{Narain}\binits{N.}} \AND
  \bauthor{\bsnm{Sidki},~\bfnm{Sa{\"{\i}}d}\binits{S.}}
(\byear{1983}).
\btitle{On the {B}urnside problem for periodic groups}.
\bjournal{Math. Z.}
\bvolume{182}
\bpages{385--388}.
\bid{doi={10.1007/BF01179757}, issn={0025-5874}, mr={0696534}}
\bptok{imsref}%
\end{barticle}
\endbibitem

\bibitem{Kai}
\begin{barticle}[mr]
\bauthor{\bsnm{Kaimanovich},~\bfnm{Vadim~A.}\binits{V.~A.}}
(\byear{2005}).
\btitle{``{M}\"unchhausen trick'' and amenability of self-similar groups}.
\bjournal{Internat. J. Algebra Comput.}
\bvolume{15}
\bpages{907--937}.
\bid{doi={10.1142/S0218196705002694}, issn={0218-1967}, mr={2197814}}
\bptok{imsref}%
\end{barticle}
\endbibitem

\bibitem{KP}
\begin{bmisc}[auto:STB|2012/08/21|06:38:10]
\bauthor{\bsnm{Kassabov},~\bfnm{M.}\binits{M.}} \AND
  \bauthor{\bsnm{Pak},~\bfnm{I.}\binits{I.}}
(\byear{2011}).
\bhowpublished{Groups of oscillating intermediate growth. Available at
  arXiv:\arxivurl{1108.0262v1}.}
\bptok{imsref}%
\end{bmisc}
\endbibitem

\bibitem{KV}
\begin{barticle}[mr]
\bauthor{\bsnm{Ka{\u\i}manovich},~\bfnm{V.~A.}\binits{V.~A.}} \AND
  \bauthor{\bsnm{Vershik},~\bfnm{A.~M.}\binits{A.~M.}}
(\byear{1983}).
\btitle{Random walks on discrete groups: Boundary and entropy}.
\bjournal{Ann. Probab.}
\bvolume{11}
\bpages{457--490}.
\bid{issn={0091-1798}, mr={0704539}}
\bptok{imsref}%
\end{barticle}
\endbibitem

\bibitem{Kes}
\begin{barticle}[mr]
\bauthor{\bsnm{Kesten},~\bfnm{Harry}\binits{H.}}
(\byear{1959}).
\btitle{Full {B}anach mean values on countable groups}.
\bjournal{Math. Scand.}
\bvolume{7}
\bpages{146--156}.
\bid{issn={0025-5521}, mr={0112053}}
\bptok{imsref}%
\end{barticle}
\endbibitem

\bibitem{LP}
\begin{bmisc}[auto:STB|2012/08/21|06:38:10]
\bauthor{\bsnm{Lee},~\bfnm{J.}\binits{J.}} \AND
  \bauthor{\bsnm{Peres},~\bfnm{Y.}\binits{Y.}}
(\byear{2009}).
\bhowpublished{Harmonic maps on amenable groups and a diffusive lower bound for
  random walks. Available at arXiv:\arxivurl{0911.0274}.}
\bptok{imsref}%
\end{bmisc}
\endbibitem


\bibitem{Nek}
\begin{bbook}[mr]
\bauthor{\bsnm{Nekrashevych},~\bfnm{Volodymyr}\binits{V.}}
(\byear{2005}).
\btitle{Self-Similar Groups}.
\bseries{Mathematical Surveys and Monographs}
\bvolume{117}.
\bpublisher{Amer. Math. Soc.}, \blocation{Providence, RI}.
\bid{mr={2162164}}
\bptok{imsref}%
\end{bbook}
\endbibitem

\bibitem{PSC1}
\begin{barticle}[mr]
\bauthor{\bsnm{Pittet},~\bfnm{Ch.}\binits{C.}} \AND
  \bauthor{\bsnm{Saloff-Coste},~\bfnm{L.}\binits{L.}}
(\byear{2000}).
\btitle{On the stability of the behavior of random walks on groups}.
\bjournal{J. Geom. Anal.}
\bvolume{10}
\bpages{713--737}.
\bid{doi={10.1007/BF02921994}, issn={1050-6926}, mr={1817783}}
\bptok{imsref}%
\end{barticle}
\endbibitem

\bibitem{PSC2}
\begin{barticle}[mr]
\bauthor{\bsnm{Pittet},~\bfnm{C.}\binits{C.}} \AND
  \bauthor{\bsnm{Saloff-Coste},~\bfnm{L.}\binits{L.}}
(\byear{2002}).
\btitle{On random walks on wreath products}.
\bjournal{Ann. Probab.}
\bvolume{30}
\bpages{948--977}.
\bid{doi={10.1214/aop/1023481013}, issn={0091-1798}, mr={1905862}}
\bptok{imsref}%
\end{barticle}
\endbibitem

\bibitem{PSC3}
\begin{barticle}[mr]
\bauthor{\bsnm{Pittet},~\bfnm{Ch.}\binits{C.}} \AND
  \bauthor{\bsnm{Saloff-Coste},~\bfnm{L.}\binits{L.}}
(\byear{2003}).
\btitle{Random walks on finite rank solvable groups}.
\bjournal{J.~Eur. Math. Soc. (JEMS)}
\bvolume{5}
\bpages{313--342}.
\bid{doi={10.1007/s10097-003-0054-4}, issn={1435-9855}, mr={2017850}}
\bptnote{check year}%
\bptok{imsref}%
\end{barticle}
\endbibitem

\bibitem{Rev}
\begin{barticle}[mr]
\bauthor{\bsnm{Revelle},~\bfnm{David}\binits{D.}}
(\byear{2003}).
\btitle{Rate of escape of random walks on wreath products and related groups}.
\bjournal{Ann. Probab.}
\bvolume{31}
\bpages{1917--1934}.
\bid{doi={10.1214/aop/1068646371}, issn={0091-1798}, mr={2016605}}
\bptok{imsref}%
\end{barticle}
\endbibitem

\bibitem{Seg}
\begin{barticle}[mr]
\bauthor{\bsnm{Segal},~\bfnm{Dan}\binits{D.}}
(\byear{2001}).
\btitle{The finite images of finitely generated groups}.
\bjournal{Proc. Lond. Math. Soc. (3)}
\bvolume{82}
\bpages{597--613}.
\bid{doi={10.1112/plms/82.3.597}, issn={0024-6115}, mr={1816690}}
\bptok{imsref}%
\end{barticle}
\endbibitem

\bibitem{Sha}
\begin{barticle}[mr]
\bauthor{\bsnm{Shannon},~\bfnm{C.~E.}\binits{C.~E.}}
(\byear{1948}).
\btitle{A mathematical theory of communication}.
\bjournal{Bell System Tech. J.}
\bvolume{27}
\bpages{379--423, 623--656}.
\bid{issn={0005-8580}, mr={0026286}}
\bptok{imsref}%
\end{barticle}
\endbibitem

\bibitem{Sid}
\begin{barticle}[mr]
\bauthor{\bsnm{Sidki},~\bfnm{Said}\binits{S.}}
(\byear{2004}).
\btitle{Finite automata of polynomial growth do not generate a free group}.
\bjournal{Geom. Dedicata}
\bvolume{108}
\bpages{193--204}.
\bid{doi={10.1007/s10711-004-2368-0}, issn={0046-5755}, mr={2112674}}
\bptok{imsref}%
\end{barticle}
\endbibitem

\bibitem{Var}
\begin{barticle}[mr]
\bauthor{\bsnm{Varopoulos},~\bfnm{Nicholas~Th.}\binits{N.~T.}}
(\byear{1985}).
\btitle{Th\'eorie du potentiel sur les groupes nilpotents}.
\bjournal{C. R. Acad. Sci. Paris S\'er. I Math.}
\bvolume{301}
\bpages{143--144}.
\bid{issn={0249-6291}, mr={0801947}}
\bptok{imsref}%
\end{barticle}
\endbibitem

\bibitem{Wil1}
\begin{barticle}[mr]
\bauthor{\bsnm{Wilson},~\bfnm{John~S.}\binits{J.~S.}}
(\byear{2004}).
\btitle{On exponential growth and uniformly exponential growth for groups}.
\bjournal{Invent. Math.}
\bvolume{155}
\bpages{287--303}.
\bid{doi={10.1007/s00222-003-0321-8}, issn={0020-9910}, mr={2031429}}
\bptok{imsref}%
\end{barticle}
\endbibitem

\bibitem{Wil2}
\begin{barticle}[mr]
\bauthor{\bsnm{Wilson},~\bfnm{John~S.}\binits{J.~S.}}
(\byear{2004}).
\btitle{Further groups that do not have uniformly exponential growth}.
\bjournal{J. Algebra}
\bvolume{279}
\bpages{292--301}.
\bid{doi={10.1016/j.jalgebra.2004.01.002}, issn={0021-8693}, mr={2078400}}
\bptok{imsref}%
\end{barticle}
\endbibitem


\bibitem{Zuk}
\begin{barticle}[auto:STB|2012/08/21|06:38:10]
\bauthor{\bsnm{Zuk},~\bfnm{Andrzej}\binits{A.}}
(\byear{2008}).
\btitle{Groupes engendr\'es par les automates}.
\bjournal{S{\'e}minaire N. Bourbaki. Ast\'erisque}
\bvolume{311}
\bpages{141--174}.
\bptok{imsref}%
\end{barticle}
\endbibitem

\end{thebibliography}
\end{document}